\documentclass[12pt]{amsart}

\usepackage{amsfonts,amsthm,amsmath,amssymb,amscd,mathrsfs}
\usepackage{graphics}
\usepackage{indentfirst}
\usepackage{cite}
\usepackage{latexsym}
\usepackage[dvips]{epsfig}
\usepackage{color}

\newtheorem{theorem}{Theorem}[section]
\newtheorem{remark}{Remark}[section]

\newtheorem{lemma}[theorem]{Lemma}

\newcommand{\bt}{\begin{theorem}}
	\newcommand{\bl}{\begin{lemma}}
		\newcommand{\el}{\end{lemma}}
	\newcommand{\et}{\end{theorem}}

\newcommand{\bn}{\begin{eqnarray}}
	\newcommand{\en}{\end{eqnarray}}
\newcommand{\bnn}{\begin{eqnarray*}}
	\newcommand{\enn}{\end{eqnarray*}}

\newcommand{\ba}{\begin{aligned}}
	\newcommand{\ea}{\end{aligned}}
\newcommand{\be}{\begin{equation}}
	\newcommand{\ee}{\end{equation}}

\newcommand{\bBV}{\boldsymbol{V}}

\newcommand{\Bu}{{\boldsymbol{u}}}
\newcommand{\Be}{{\boldsymbol{e}}}

\newcommand{\bBu}{\bar{{\boldsymbol{u}}}}

\newcommand{\Bx}{{\boldsymbol{x}}}
\newcommand{\By}{{\boldsymbol{y}}}

\newcommand{\BP}{{\boldsymbol{\Psi}}}

\begin{document}
	
	\title[Liouville-type theorems]
	{Liouville-type theorems for steady solutions to the Navier-Stokes system in a slab}

	\author{Jeaheang Bang}
	\address{Department of Mathematics, The University of Texas at San Antonio}
	\email{jeaheang.bang@utsa.edu}

	\author{Changfeng Gui}
	\address{Department of Mathematics, The University of Texas at San Antonio}
	\email{Changfeng.Gui@utsa.edu}

	\author{Yun Wang}
	\address{School of Mathematical Sciences, Center for dynamical systems and differential equations, Soochow University, Suzhou, China}
	\email{ywang3@suda.edu.cn}

	\author{Chunjing Xie}
	\address{School of mathematical Sciences, Institute of Natural Sciences,
		Ministry of Education Key Laboratory of Scientific and Engineering Computing,
		and CMA-Shanghai, Shanghai Jiao Tong University, 800 Dongchuan Road, Shanghai, China}
	\email{cjxie@sjtu.edu.cn}

	\begin{abstract}
		Liouville-type theorems for the steady incompressible Navier-Stokes system are investigated for solutions in a three-dimensional slab with either no-slip boundary conditions or periodic boundary conditions. When the no-slip boundary conditions are prescribed, we prove that any bounded solution is trivial if it is axisymmetric or $ru^r$ is bounded, and that general three-dimensional solutions must be Poiseuille flows when the velocity is not big in $L^\infty$ space.  When the periodic boundary conditions are imposed on the slab boundaries, we prove that the bounded solutions must be constant vectors  if either  the swirl or radial velocity is independent of the angular variable, or $ru^r$ decays to zero as $r$ tends to infinity. The proofs are based on the fundamental structure of the equations and  energy estimates.  The key technique is to establish a Saint-Venant type estimate that characterizes the growth of Dirichlet integral of nontrivial solutions.
	\end{abstract}

	\keywords{Liouville-type theorem, steady Navier-Stokes system,  slab, no-slip boundary condition, periodic}
	\subjclass[2010]{
	35B53,	 35Q30, 35J67, 35B10, 76D05}

	
	\maketitle

	\section{Introduction and Main Results}
	It is well known that every bounded harmonic function on $\mathbb{R}^n$ is constant. This is the classical Liouville theorem for Laplace equation. Such a kind of results was later generalized to many other partial differential equations. It has played an important role in  analyzing singularity and classification of solutions for PDEs, etc. In this paper, we are interested in the Liouville-type theorem for solutions of steady incompressible Navier-Stokes system,
	\be \label{SNS}
	\left\{ \ba
	& -\Delta \Bu + (\Bu \cdot \nabla )\Bu + \nabla P  = 0, \ \ \ \ \ &\mbox{in}\ \Omega, \\
	& \nabla \cdot \Bu =0,  \ \ \ \ \ &\mbox{in}\ \Omega,  \\
	\ea \right.
	\ee
	where  $\Bu=(u^1, u^2, u^3)$ is the velocity field, and $\Omega$ is the slab $\mathbb{R}^2 \times [0, 1]$ with no-slip boundary conditions
	\begin{equation}\label{noslipBC}
		\Bu=0 \ \ \ \ \ \ \text{at}\,\, x_3=0\,\, \text{and}\,\, 1,
	\end{equation}
	or with the  periodic boundary condition in $x_3$.
	In the latter case, we may denote  the domain  by $\mathbb{R}^2\times \mathbb{T}$ and regard  the flows as periodic flows in one direction.

	Note that the solutions of \eqref{SNS} in $\mathbb{R}^2\times \mathbb{T}$ can also be regarded as solutions in the whole $\mathbb{R}^3$.
	The existence of solutions for \eqref{SNS} in an exterior domain $\Omega$ satisfying
	\be \label{Dirichletintegral}
	\int_\Omega |\nabla \Bu|^2dx  < \infty
	\ee
	was studied by Leray in the pioneering work \cite{Leray}.  A solution to \eqref{SNS} in an exterior domain or the whole space is called a $D$-solution if it satisfies \eqref{Dirichletintegral}.  With the aid of the maximum principle for the vorticity, a Liouville theorem for $D$-solutions of two-dimensional steady  Navier-Stokes system in the whole plane was established in \cite{GW}. A longstanding open problem is whether a three-dimensional $D$-solution in an exterior domain with homogeneous boundary conditions or in the whole space must equal to 0 when it tends to zero at far field (cf. \cite{Galdi}).  When $\Omega = \mathbb{R}^3$, there are a lot of studies for this problem.  It was proved in  \cite[Theorem
	X.9.5]{Galdi} that the $D$-solution must be zero provided  $\Bu \in L^{\frac{9}{2} }(\mathbb{R}^3)$. There are quite a lot of progress along this direction in recent years, one may refere to \cite{Chae-Wolf, KTW, Seregin, Chae0} and references therein.
	
	The cylindrical coordinates are convenient to study  the Navier-Stokes system. Here we first introduce the cylindrical coordinates $(r,\theta, z)$  defined as follows
	\begin{equation}\label{cylindrcord}
		x_1=r \cos \theta, \quad x_2=r\sin\theta, \quad\text{and}\,\, x_3=z;
	\end{equation}
	And we write the components of the velocity $\Bu$ as follows: $\Bu=u^r \Be_r +u^\theta \Be_\theta +u^z \Be_z$ where $u^r$, $u^\theta$, and $u^z$ are called radial, swirl, and axial velocity, respectively, with
	\[
	\Be_r=(\cos\theta, \sin \theta, 0),\quad \Be_\theta =(-\sin\theta, \cos\theta, 0), \quad \text{and}\,\, \ \Be_z=(0, 0,1).
	\]
	
	We sometimes use Cartesian coordinates $(x_1,x_2,x_3)$ with $(u^1,u^2,u^3)$ and, other times, use  cylindrical coordinates $(r,\theta,z)$ with $(u^r,u^\theta,u^z)$. If no confusion arises, we may write $x_3$ and $z$ interchangeably to indicate the third coordinate. Likewise, we may use $u^3$ and $u^z$ interchangeably to represent the third component of the velocity $\Bu$.
	
	A solution $\Bu$ is called an axisymmetric solution if all the components $u^r, u^\theta, u^z$ are independent of $\theta$.
	In the cylindrical coordinates, the Navier-Stokes system \eqref{SNS} can be written as
	\be \label{NScylindrical}
	\left\{
	\begin{aligned}
		&\left(u^r \partial_r + \frac{u^\theta}{r}  \partial_\theta + u^z \partial_z\right)u^r - \frac{(u^\theta)^2}{r} + \frac{2}{r^2} \partial_\theta u^\theta + \partial_r P = \left(\Delta_{r,\theta, z} - \frac{1}{r^2} \right) u^r,\\
		&\left(u^r \partial_r + \frac{u^\theta }{r}  \partial_\theta + u^z \partial_z\right)u^\theta +
		\frac{u^\theta u^r}{r} - \frac{2}{r^2} \partial_\theta u^r + \frac{1}{r} \partial_\theta P =  \left(\Delta_{r,\theta, z} - \frac{1}{r^2} \right) u^\theta,\\
		&\left(u^r \partial_r + \frac{u^\theta}{r} \partial_\theta + u^z \partial_z\right)u^z + \partial_z P =\Delta_{r,\theta, z} u^z,\\
		&\partial_r u^r +\frac{1}{r}\partial_\theta u^\theta+\partial_z u^z +\frac{u^r}{r}=0,
	\end{aligned}
	\right.
	\ee
	where
	\[
	\Delta_{r, \theta, z}=\partial_r^2 + \frac{1}{r} \partial_r + \frac{1}{r^2} \partial_\theta^2 + \partial_z^2.
	\]

	The Liouville-type theorems for axisymmetric D-solutions were established when the pointwise behavior for the velocity field or vorticity is prescribed at far field, see \cite{KTW,Zhao, CPZZ} and references therein.
	
	As pointed out in \cite{Tsai}, a strong version of Liouville-type theorems for the steady Navier-Stokes system in the whole space was conjectured by Seregin and Sverak: whether the  solutions
	$\Bu \in H_{loc}^1(\mathbb{R}^3) \cap L^\infty(\mathbb{R}^3)$
	must be constant vectors. A significant result in \cite{KNSS,CSYT1, CSYT2} asserts that every bounded axisymmetric steady solution with type I singularity for Navier-Stokes system  is trivial. This Liouville-type theorem also holds for bounded two-dimensional solutions for steady Navier-Stokes system in the whole plane. Furthermore,  the analysis on ancient solutions of Navier-Stokes system in \cite{LRZ} asserts that the bounded axisymmetric steady solution in $\mathbb{R}^2 \times \mathbb{T}$ must be zero if the swirl velocity also decays as $r^{-1}$ at far field.

	Very recently, the flows in a slab were studied in \cite{CPZ,CPZZ}.  It was  proved  that the  smooth solution $\Bu$ of equation \eqref{SNS}  must be $0$ if $\Bu$ satisfies either the no-slip boundary conditions \eqref{noslipBC} together with \eqref{Dirichletintegral}
	or that $\Bu$ is axisymmtric, periodic in the axial direction, and satisfies \eqref{Dirichletintegral}.
	The results for the bounded periodic solutions were improved in \cite{Pan}  where both condition \eqref{Dirichletintegral} and independence of $\theta$ for the swirl velocity
	$u^\theta$ are needed. The Liouville-type theorem for flows in a slab with both no-slip boundary conditions and periodic conditions was obtained under refined conditions in \cite{Tsai}.

	In this paper, we focus on  the Liouville-type theorem for flows in a slab.
	Our first result is on the axisymmetric solutions of steady Navier-Stokes system in a slab with no-slip boundary conditions.

	\bt \label{mainresult1} Let $\Omega =\mathbb{R}^2 \times (0,1)$ and $\Bu$ be a smooth axisymmetric solution of the Navier-Stokes system \eqref{SNS} with no-slip boundary conditions \eqref{noslipBC}.
	Then  $\Bu \equiv 0$ if
	
	either (a)
	\be \label{growth}
	\varliminf_{R\rightarrow + \infty}  R^{-4}  \mathcal{E}(R) = 0
	\ee
		where \begin{equation}\label{defER}
		\mathcal{E}(R)=\displaystyle \int_0^1 \int_{\{x_1^2+x_2^2 < R^2\}} |\nabla \Bu|^2 dx_1dx_2dx_3,
			\end{equation}
	or (b)
	 \begin{equation} \label{growth-coro1}
	\lim_{R\rightarrow + \infty} R^{-1}  \sup_{z \in [0, 1]} |\Bu (R, z) |        =0.
		\end{equation}
	\et
	
	There are a few remarks in order.
	\begin{remark}
		Theorem \ref{mainresult1} implies that any nontrivial axisymmetric solution (if it exists) to the homogeneous Navier-Stokes equations in a slab should grow at least linearly.
	\end{remark}
	
	\begin{remark}
		The proof of Theorem \ref{mainresult1} also shows that
		if $\Bu$ is a nontrivial solution to the homogeneous Stokes equations in $\mathbb{R}^2 \times (0, 1)$ with no-slip boundary conditions, then
		$\mathcal{E}(R)$
		must grow exponentially. This phenomenon is quite similar to that for Laplace equation in a slab.
	\end{remark}
	
	Next, we state our results on general three-dimensional flows in a slab with certain asymptotic constraints.
	\bt \label{mainresult4}Let $\Bu$ be a smooth solution to the Navier-Stokes system \eqref{SNS} in $\Omega=\mathbb{R}^2\times (0,1)$ with no-slip boundary conditions \eqref{noslipBC}.
	Then there exists an absolute number $\alpha \in (0,1)$ such that $\Bu\equiv 0$ if one of the following conditions holds:
	
	$(a)$\
	\be \label{growth-4}
	\varliminf_{R\rightarrow + \infty}  R^{-\alpha}  \mathcal{E}(R)  =  0
	\ee
	where $\mathcal{E}(R)$ is defined in \eqref{defER};
	
		$(b)$\  $\Bu$ satisfies for some $\beta \in [0,\frac{\alpha}{2}]$,
	\begin{equation}\label{condthm1.2}
		\lim_{r\to\infty} r^{-\beta }\sup_{\substack{{\theta\in [0,2\pi)}\\{z\in [0,1]}}} |\Bu (r, \theta, z)| =0\quad\text{and}\quad  \sup_{(r, \theta, z)\in \Omega}r^{1+\beta -\frac{\alpha}{2}}| u^r(r, \theta, z)|<\infty.
	\end{equation}
	\et

	\begin{remark} A direct consequence of Theorem \ref{mainresult4} is as follows.
		Let $\Bu$ be a smooth solution to the Navier-Stokes system \eqref{SNS} such that
		\be \nonumber
		\int_0^1 \int_{\mathbb{R}^2 } |\nabla \Bu |^2 dx < \infty.
		\ee
		Then $\Bu \equiv 0$. This is exactly  \cite[Theorem 1.1]{CPZZ}, which asserts that a homogeneous D-Solution to the steady incompressible Navier-Stokes system in a slab with no-slip boundary conditions must be 0. Hence Theorem \ref{mainresult4} improves the results obtained in \cite[Theorem 1.1]{CPZZ}.
	\end{remark}

	\begin{remark} As long as there is a number  $\beta\in [0,\frac{\alpha}{2}]$ such that the conditions \eqref{condthm1.2} are satisfied, $\Bu$ must be $0$.
	In particular, we have $\Bu\equiv 0$ when both $\Bu$ and $ru^r$ are uniformly bounded, where the conditions \eqref{condthm1.2} are guaranteed by $\beta=\frac{\alpha}{2}$.
	\end{remark}
	
	\begin{remark} \label{RemExpo}
		The conditions in Case (b) of Theorem \ref{mainresult4} also include another special case that
		$\lim_{r \to\infty} \sup_{\theta, z} |\Bu(r, \theta, z)| =0$ and $r^{1-\frac{\alpha}{2}} u^r$ is uniformly bounded, where the conditions \eqref{condthm1.2} are satisfied with $\beta=0$.
		
	\end{remark}
	
	\begin{remark}\label{lbb}
	It should be interesting to get the optimal exponent $\alpha$ in \eqref{growth-4}. From the proof for Theorem \ref{mainresult4}, this optimal exponent should be related to the best constant for the Bogovskii map in Lemma \ref{Bogovskii}.
		\end{remark}

	
	\begin{remark}
		If $\Bu$ is only bounded, then the trivial solution $\Bu=0$ may  not be the only solution for Navier-Stokes system in a slab with no-slip boundary conditions. For example, $\Bu= (c_1x_3(1-x_3), c_2x_3(1-x_3), 0)$ is also a  bounded solution of \eqref{SNS} in a slab with no-slip boundary conditions. These solutions are called Poiseuille flow in a layer. It is a conjecture whether the bounded solution to the Navier-Stokes system must be a Poiseuille flow.
	\end{remark}
	
	In fact,
	if $\Bu$ is not too big,  we have the following Liouville-type theorem for general three-dimensional flows in a slab.

	\bt\label{mainresult5} Let $\Bu$ be a smooth solution to the Navier-Stokes system \eqref{SNS} in a slab $\mathbb{R}^2\times (0,1)$ with no-slip boundary conditions \eqref{noslipBC}. Then $\Bu $ must be a Poiseuille flow of form  $(c_1x_3(1-x_3), c_2x_3(1-x_3), 0)$ if
	\be \label{Poiseuille-1}
	\| \Bu \|_{L^\infty(\Omega)} <\pi.
	\ee
	\et

	\begin{remark}
	 The assumption on the upper bound of $\Bu$ in Theorem \ref{mainresult5} may not be optimal. The current bound comes from the estimate \eqref{main5-5}. 	It is still an open problem whether Poiseuille flow is the unique bounded solution when $\|\Bu\|_{L^\infty(\Omega)}$ is not small. When the domain is a strip in $\mathbb{R}^2$, it relates to the resolution of the famous Leray problem for flows in infinitely long nozzles, which is still a longstanding open problem. One may refer to \cite{LS, Galdi, WX} for more discussion about this problem.
	\end{remark}
	
	\begin{remark} The proof of Theorem \ref{mainresult5} shows that  $\Bu$ must be a Poiseuille flow, i.e., $\Bu = (c_1x_3(1-x_3), c_2x_3(1-x_3), 0)$, provided $\Bu$ is a bounded smooth solution to the homogeneous Stokes equations in $\mathbb{R}^2 \times [0,1]$. Hence the Liouville-type theorem for the Stokes system in a slab can be established even when the assumption \eqref{Poiseuille-1} is removed.
	\end{remark}
	
	
	Finally, we state our results on Liouville-type theorems for the steady Navier-Stokes system in a slab with periodic boundary conditions.
	\bt \label{corollary2}Let $\Bu$ be a bounded smooth solution to the Navier-Stokes system \eqref{SNS} in $\mathbb{R}^2\times \mathbb{T}$.
	Then $\Bu$ must be a constant vector provided that one of the following conditions holds:
	\begin{enumerate}
		
		\item[(a)] $u^\theta$ is axisymmetric, i.e., $u^\theta$ is independent of $\theta$;
		\item[(b)] $u^r$ is axisymmetric, i.e.,  $u^r$ is independent of $\theta$;
		
		\item[(c)] $ru^r$ converges to $0$, as $r \rightarrow +\infty$.

\item[(d)] $\|\Bu\|_{L^\infty(\Omega)} <2\pi. $

	\end{enumerate}
	Furthermore, in all cases (a), (b), and (c), the only nonzero component of the velocity field must be $u^z$, i.e., the constant vector $\Bu$ must be of the form $(0,0, c)$.
	\et
	
	\begin{remark}
		It is clear that axisymmetric solutions satisfy the condition in case (a) of Theorem \ref{corollary2}.  Therefore, any bounded axisymmetric solution periodic in the axial direction must be of the form $(0, 0, c)$.
	\end{remark}

	
	
	Here we illustrate the major methods of the paper.
	The key idea of this paper is to establish various differential inequalities for the Dirichlet integral of $\Bu$ over a finite part $\Omega_R$. The idea dates back
	to the proof of Saint-Venant's principle for solutions to equations of elasticity  in \cite{Toupin, Knowles}.  This idea was also generalized to deal with linear elliptic equations in \cite{OY} where the growth of Dirichlet integral for solutions was established. When adapting the same idea to Stokes or Navier-Stokes system in a pipe, the presence of pressure term causes  serious difficulties. Inspired by the work \cite{LS}, we develop a new form of the  Bogovskii map in Lemma \ref{Bogovskii}  to estimate the pressure term in a careful and efficient way, utilizing the structure of the Navier-Stokes equations. In particular, we make  frequent use of the improved estimates in \eqref{BetterEst}.
	We would like to mention that the stream function can be used to deal with axisymmetric flows, see \cite{Horgan-Wheeler}. However, it does not seem to work well with general three-dimensional flows.
	
	The organization for the rest of the paper is as follows. Some preliminary results on Bogovskii map and differential inequalities are presented in Section \ref{Sec2}. The proof of axisymmetric solution in a slab with no-slip boundary conditions is given in Section \ref{Sec3}. In Section \ref{Sec4}, we give the proof for Liouville-type theorem for general three-dimensional solutions in a slab with no-slip boundary conditions. Finally, the Liouville-type theorem for the steady Navier-Stokes system in a periodic slab is established in Section \ref{Sec5}.
	
	\section{Preliminaries}\label{Sec2}
	In this section, we give some preliminaries. First, we introduce the following notations. Define
	\[
	L^p_0(\Omega)=\left\{f: \ \ f\in L^p(\Omega), \ \ \int_\Omega f dx =0 \right\}.
	\]
	For any $R\geq 2$, denote
	$D_R = (R-1, R)\times (0, 1)$, $\mathcal{D}_R = (R-1, R) \times (0, 2\pi) \times (0, 1)$,  $\Omega_R = B_R \times (0, 1)$, and $\mathcal{O}_R = (B_R \setminus \overline{B_{R-1}}) \times (0, 1)$, where $B_R= \{(x_1,x_2)\in \mathbb{R}^2: x_1^2+x_2^2 <R^2\}$. For any $\Bx\in \mathbb{R}^3$, define $\mathcal{B}_r(\Bx)=\{\By\in \mathbb{R}^3: |\By-\Bx|<r\}$.
	In the rest of the paper, $\varphi_R(r)$ denotes the smooth cut-off function satisfying
	\be \label{cut-off}
	\varphi_R(r) = \left\{ \ba
	&1,\ \ \ \ \ \ \ \ \ \ r < R-1, \\
	&R-r,\ \ \ \ R-1 \leq r \leq R, \\
	&0, \ \ \ \ \ \ \ \ \ \ r > R.
	\ea  \right.
	\ee
	
	Here we introduce the Bogovskii map, which gives a solution to the divergence equations. The general form is due to Bogovskii\cite{Bogovskii}, see also  \cite [Section III.3]{Galdi} and \cite[Section 2.8]{Tsai-Book}.
	\begin{lemma}\label{Bogovskii}
		Let $\Omega$ be a bounded Lipschitz domain in $\mathbb{R}^n$ with $n\geq 2 $.  For any $q\in (1, \infty)$, there is a linear map $\boldsymbol{\Phi}$ that maps a scalar function $f\in L^q_0(\Omega)$ to a vector field $\bBV = \boldsymbol{\Phi} f \in W_0^{1, q}(\Omega; \mathbb{R}^n)$ satisfying
		\be \nonumber
		{\rm div}~\bBV = f \ \ \ \text{and} \quad \|\bBV\|_{W_0^{1, q}(\Omega)} \leq C (\Omega, q) \|f\|_{L^q(\Omega)}.
		\ee
		In particular, we have the following results.
		\begin{enumerate}
			\item    For any $f \in L^2_0(D_R)$,
			the  vector valued function $\bBV = \boldsymbol{\Phi} f \in H_0^1(D_R; \mathbb{R}^2)$ satisfies
			\be \nonumber
			\partial_r V^r + \partial_z V^z =f \ \  \mbox{in}\,\, D_R
			\quad
			\text{and}
			\quad
			\|\tilde{\nabla } \bBV\|_{L^2(D_R)}
			\leq C \|f\|_{L^2(D_R)},
			\ee
			where $\tilde{\nabla } = (\partial_r, \ \partial_z ) $ and $C$ is some constant independent of $R$.
			
			\item  For any $f \in L^2(\mathcal{D}_R)$,
			the   vector valued function $\bBV = \boldsymbol{\Phi} f \in H_0^1( \mathcal{D}_R; \mathbb{R}^3)$ satisfies
			\be \nonumber
			\partial_r V^r + \partial_\theta V^\theta +  \partial_z V^z =f \ \   \mbox{in}\,\, \mathcal{D}_R
			\quad
			\text{and}
			\quad
			\|\bar{\nabla } \bBV\|_{L^2(\mathcal{D}_R)}
			\leq C \|f\|_{L^2( \mathcal{D}_R)},
			\ee
			where $\bar{\nabla } = (\partial_r, \ \partial_\theta, \ \partial_z ) $ and $C$ is some constant independent of $R$.
		\end{enumerate}
	\end{lemma}
	
	The constants in the estimates of parts (1) and (2) do not depend on $R$; indeed, for $\boldsymbol{\Phi}$ given in $D_R=D_2$, we can define $\boldsymbol{\Phi}_R$ in $D_R$ for general $R$ as follows. For $f\in L^2_0 (D_R), $ it holds that $(\tau_{-R+2}\circ f)(r,z):= f(r-R+2,z)\in L^2_0 (D_2)$ and $V=\boldsymbol{\Phi} (\tau_{-R+2}\circ f) \in H^1_0 (D_2; \mathbb{R}^2)$.  Then $(\tau_{R-2}\circ) V(r,z)= V(r+R-2,z)\in H^1_0 (D_R; \mathbb{R}^2)$ and we can define $\boldsymbol{\Phi}_R f= \tau_{R-2}V$.
	
	We will apply the Bogovskii map $\boldsymbol{\Phi}$ in a non-standard way. In a domain $\mathcal{O}_R$, one can solve the equation
	\begin{equation} \label{divEq}
		\mathrm{div} \thinspace \bBV = f \quad \text{in }\mathcal{O}_R
	\end{equation}
	by utilizing the Bogovskii map $\boldsymbol{\Phi}$ in a standard way, which yields the estimate
	\begin{equation} \label{BogoEst}
		\|\nabla \bBV\|_{L^2 (\mathcal{O}_R)} \leq CR \|f\|_{L^2 (\mathcal{O}_R)}.
	\end{equation}
	On the other hand, one can solve the same equation \eqref{divEq} with different estimates as follows. Define a function $g$ by $g=rf(r,\theta,z)$ in $\mathcal{D}_R$ and a vector field $\boldsymbol{W}=(W^r, W^\theta, W^z)$ by
	\begin{equation}
		W^r=rV^r, \quad W^\theta= V^\theta,\quad W^z=rV^z \quad \text{in }\mathcal{D}_R.
	\end{equation}
	Then using cylindrical coordinates, one can see that equation \eqref{divEq} is equivalent to the equation
	\begin{equation} \label{divEqD}
		\frac{\partial W^r}{\partial r}+
		\frac{\partial W^\theta}{\partial \theta}
		+\frac{\partial W^z}{\partial z}=g \quad \text{in }\mathcal{D}_R.
	\end{equation}
	Note equation \eqref{divEq} holds in $\mathcal{O}_R$ whereas equation \eqref{divEqD} holds in $\mathcal{D}_R$. Applying the Bogovskii map to solve the equation \eqref{divEqD} now yields an estimate of the gradient $(\partial_r, \partial_\theta, \partial_z )\boldsymbol{W}$ in terms of $g$, which, in turn, gives the following estimates:
	\begin{equation} \label{BetterEst}
		\begin{aligned}
				\left\|
			\frac{1}{r} \partial_\theta V^z
			\right\|_{L^2(\mathcal{O}_R)} + 	\left\|\frac{1}{r}\partial_\theta V^r
			\right\|_{L^2(\mathcal{O}_R)}
			\leq \,& CR^{-1}\| f
			\|_{L^2(\mathcal{O}_R)},\\		
			\|(\partial_r, \partial_z)V^r\|_{L^2(\mathcal{O}_R)} +\|
			(\partial_r, \partial_z)V^z
			\|_{L^2(\mathcal{O}_R)}+ \left\| \frac{1}{r} \partial_\theta V^\theta
			\right\|_{L^2(\mathcal{O}_R)}
			\leq \,& C\|f\|_{L^2(\mathcal{O}_R)},\\
			\|(\partial_r, \partial_z) V^\theta
			\|_{L^2(\mathcal{O}_R)}
			\leq \,& CR\|f
			\|_{L^2(\mathcal{O}_R)}.
				\end{aligned}
		\end{equation}
	One can compare estimates \eqref{BetterEst} to \eqref{BogoEst}. All estimates in \eqref{BetterEst} have constants with better behavior with respect to $R$ than \eqref{BogoEst} except the third one in \eqref{BetterEst}. We will use the Bogovskii map $\boldsymbol{\Phi}$ in this way rather than the standard way.
	
	The second lemma is used to characterize the growth of functions which satisfy various differential inequalities.
	\begin{lemma}\label{PL}
		Let $\phi(t)$ be a nondecreasing nonnegative function  and $t_0>1$ be a fixed constant. Suppose that $\phi(t)$ is not identically zero.
		
		(a)  If $\phi(t)$  satisfies
		\be \label{lemma2-1}
		\phi(t) \leq C_1 \phi^{\prime}(t) + C_2 \left[ \phi^{\prime}(t) \right]^{\frac32} \ \ \ \text{for any}\,\, t\geq t_0,
		\ee
		then
		\be \label{lemma2-2}
		\varliminf_{t \rightarrow + \infty} t^{-3} \phi(t) > 0.
		\ee
		
		(b) If $\phi(t)$ satisfies
		\be \label{lemma2-3}
		\phi(t) \leq C_3 \phi^{\prime}(t) + C_4 t^{-\frac12} \left[ \phi^{\prime}(t) \right]^{\frac32} \ \ \ \text{for any}\,\, t\geq t_0,
		\ee
		then
		\be \label{lemma2-4}
		\varliminf_{t \rightarrow + \infty} t^{-4} \phi(t) > 0.
		\ee
		
		(c) If $\phi(t_0)>0$ and there exist constants $C_1>1$ and $C_2>0$ such that $\phi(t)$  satisfies
		\be
		\phi(t) \leq C_1 t \phi^{\prime}(t) + C_2 t \left[ \phi^{\prime}(t) \right]^{\frac32} \ \ \ \text{for any}\,\, t\geq t_0,
		\ee
		then for any fixed $0<\alpha<\frac{1}{C_1}$, one has
		\be
		\phi(t) \geq  Ct^{\alpha}\quad \text{for any}\,\, t\geq t_0,
		\ee
		where $C=\min\left\{\frac{\phi(t_0)}{2}t_0^{-\alpha},\frac{1}{4\alpha}\left(\frac{1}{\alpha C_1}-1\right)^2  \frac{C_1^2}{C_2^2}\right\}$.
	\end{lemma}
	The proof for Case (a) of  Lemma \ref{PL} can be found in \cite{LS}. For readers' convenience, we give a short proof.
	\begin{proof} We prove the lemma case by case.

{\it Step 1. Proof for Case (a)}.	Assume that $\phi(t_1) > 0$ for some $t_1 \geq t_0$. According to \eqref{lemma2-1}, it holds that
		\be
		C_1 \phi^{\prime}(t) + C_2 \left[ \phi^{\prime}(t) \right]^{\frac32} \geq \phi(t) \geq  \phi (t_1) > 0 \ \ \ \ \ \text{for any}\,\, \ t \geq t_1.
		\ee
		Hence there exists a positive constant $\alpha_1$ such that
		\be \label{lemma2-5}
		\phi^{\prime} (t) > \alpha_1 \ \ \ \ \ \ \text{for any}\,\, \ t \geq t_1.
		\ee
		Taking \eqref{lemma2-5} into \eqref{lemma2-1} yields
		\be \label{lemma2-7}
		\phi(t) \leq ( C_1 \alpha_1^{-\frac12} + C_2 )\left[ \phi^{\prime} (t) \right]^{\frac32}\ \ \ \ \text{for any}\,\,\ t \geq t_1,
		\ee
		which implies \eqref{lemma2-2} finishes the proof for Case (a).
		
	{\it Step 2. Proof for Case (b)}. If $\phi(t)$ satisfies \eqref{lemma2-3}, as proved above, $\phi(t)$ grows at least in cubic order. Hence there exist a $t_2> t_0$ and a constant $\alpha_2$,  such that
		\be \label{lemma2-9}
		C_3 \phi^{\prime}(t) + C_4 t^{-\frac12} \left[ \phi^{\prime}(t) \right]^{\frac32}\geq \alpha_2 t^3 \ \ \ \ \text{for} \ \ \ t\geq t_2.
		\ee
		Hence for $t\geq t_2$,  either $C_3\phi'(t)\geq \frac{\alpha_2}{2}t^3$ or $C_4t^{-\frac{1}{2}}(\phi'(t))^{\frac{3}{2}}\geq \frac{\alpha_2}{2}t^3$. Therefore, there exists a positive constant $\alpha_3>0$ such that
		\be \label{lemma2-10}
		\phi^{\prime}(t) \geq \alpha_3 t^{\frac{7}{3}} \quad  \text{for} \ \ \ t\geq t_2.
		\ee
		Hence one has
		\be \label{lemma2-11}
		\phi(t) \leq C_3 \alpha_3^{-\frac12}   t^{-\frac{7}{6}} \left[ \phi^{\prime}(t) \right]^{\frac32}  + C_4 t^{-\frac12} \left[ \phi^{\prime}(t) \right]^{\frac32} \leq C_5 t^{-\frac12} \left[ \phi^{\prime}(t) \right]^{\frac32},
		\ee
		which implies \eqref{lemma2-4} and completes the proof for Case (b).
		
		{\it Step 3. Proof for Case (c)}.	Note that $Ct_0^{\alpha}\leq \frac{1}{2}\phi(t_0)$. Hence  there exists a $\delta>0$ such that
		\[
		\phi(t)\geq Ct^\alpha \quad \text{for any}\,\, t\in [t_0, t_0+\delta).
		\]
		Define
		\[
		\bar{t}={\sup}\{s: \phi(\tau)\geq C\tau^\alpha \,\, \text{for any}\,\, \tau\in (t_0, s)\}.
		\]
		Clearly, $\bar{t}\geq t_0+\delta>1$. Suppose that $\bar{t}<\infty$.  Then one has $\phi(\bar{t})=C\bar{t}^\alpha$.
		
		If $\phi'(\bar{t})>\frac{1}{4}\left(\frac{1}{\alpha C_1}-1\right)^2\left(\frac{C_1}{C_2}\right)^2$, then there exists a $\tilde{\delta}>0$ such that
		\[
		\phi'({t})>\frac{1}{4}\left(\frac{1}{\alpha C_1}-1\right)^2\left(\frac{C_1}{C_2}\right)^2\quad \text{for any}\,\, t\in [\bar{t}, \bar{t}+\tilde{\delta}).
		\]
		Hence for any $t\in (\bar{t}, \bar{t}+\tilde{\delta})$, one has
		\begin{equation*}
			\begin{aligned}
				\phi(t)=&\phi(\bar{t})+\int_{\bar{t}}^t \phi'(s)ds> \phi(\bar{t})+\int_{\bar{t}}^t\frac{1}{4}\left(\frac{1}{\alpha C_1}-1\right)^2 \left(\frac{C_1}{C_2}\right)^2 ds
				\geq  \phi(\bar{t})+\int_{\bar{t}}^t \alpha C ds\\
				=& C\bar{t}^\alpha +\int_{\bar{t}}^t C \alpha ds\geq  C\bar{t}^\alpha+\int_{\bar{t}}^t C \alpha s^{\alpha-1} ds =Ct^\alpha.
			\end{aligned}
		\end{equation*}
		This contradicts with the definition of $\bar{t}$.
		
		If $\phi'(\bar{t})\leq \frac{1}{4}\left(\frac{1}{\alpha C_1}-1\right)^2\left(\frac{C_1}{C_2}\right)^2$, then there exists a $\hat{\delta}>0$ such that
		\[
		\phi'({t})<\left(\frac{1}{\alpha C_1}-1\right)^2\left(\frac{C_1}{C_2}\right)^2\quad \text{for any}\,\, t\in [\bar{t}, \bar{t}+\hat{\delta}).
		\]
		Thus for any $t\in (\bar{t}, \bar{t}+\hat{\delta})$, one has
		\begin{equation*}
			\begin{aligned}
				\phi(t)\leq &C_1 t\phi'({t})+C_2t\Big(\phi'(t)\Big)^{\frac{3}{2}}<C_1 t\phi'(t)+C_2 t\phi'(t) \left(\frac{1}{\alpha C_1}-1\right)\frac{C_1}{C_2}\leq \frac{1}{\alpha}  t\phi'(t).
			\end{aligned}
		\end{equation*}
		This implies
		\begin{equation*}
			\ln\frac{\phi(t)}{\phi(\bar{t})}\geq \alpha \ln\frac{t}{\bar{t}}.
		\end{equation*}
		Hence for any $t\in (\bar{t}, \bar{t}+\hat{\delta})$, it holds that
		\[
		\phi(t)\geq \phi(\bar{t})\left(\frac{t}{\bar{t}}\right)^{\alpha}=C\bar{t}^\alpha \bar{t}^{-\alpha}t^\alpha= Ct^\alpha.
		\]
		This also leads to a  contradiction with the definition of $\bar{t}$. Therefore, $\bar{t}=+\infty$ and the proof for Case (c) is completed.		
	\end{proof}

	The following lemma shows that the bounded solutions of Navier-Stokes system in a slab with no-slip boundary conditions or periodic boundary conditions have bounded gradient.

	\begin{lemma}\label{bounded-gradient}
		Let $\Bu$ be a bounded smooth solution to the Navier-Stokes system \eqref{SNS} in $\mathbb{R}^2\times (0,1)$ supplemented with no-slip boundary conditions or in $\mathbb{R}^2\times \mathbb{T}$. Then $\nabla \Bu$, $\nabla^2 \Bu$, and $\nabla P$ are also uniformly bounded.
	\end{lemma}
	\begin{proof} The proof is divided into two steps.
		
		{\it	Step 1. interior regularity.}	According to  \cite[Theorem IV.4.1, Theorem IV.4.4, and Remark IV.4.2]{Galdi}, it holds that  for any $\Bx\in \mathbb{R}^2\times [\frac{1}{8}, \frac{7}{8}]$,
		\be \label{bound-gradient-1}
		\|\nabla \Bu \|_{L^4(\mathcal{B}_{\frac{5}{64}}(\Bx) )} \leq C \| \Bu\|_{L^8(\mathcal{B}_{\frac{3}{32}}(\Bx))}^2 +C \|\Bu \|_{L^4(\mathcal{B}_{\frac{3}{32}}(\Bx))} \leq C .
		\ee
		Moreover, one has
		\be \label{bound-gradient-2}
		\|\nabla \Bu \|_{W^{1, 4} (\mathcal{B}_{\frac{1}{16}}(\Bx)) } \leq C \| \nabla \Bu \|_{L^4(\mathcal{B}_{\frac{5}{64}}(\Bx)) } \| \Bu \|_{L^\infty(\mathcal{B}_{\frac{5}{64}}(\Bx))} + C \|\Bu \|_{W^{1, 4}(\mathcal{B}_{\frac{5}{64}}(\Bx))} \leq C .
		\ee
     \be \label{bound-gradient-3}
     \|\nabla^2 \Bu \|_{W^{1, 4} ( \mathcal{B}_{\frac{1}{32}} (\Bx)) } \leq C \| \Bu \cdot \nabla \Bu \|_{W^{1, 4}(\mathcal{B}_{\frac{1}{16}} (\Bx) )} 
     + C \|\Bu \|_{W^{2, 4} ( \mathcal{B}_{\frac{1}{16}} (\Bx))} \leq C. 
     \ee

		{\it Step 2. boundary regularity.} According to  \cite[Theorem IV.5.1, Theorem IV.5.3, and Remark IV.5.2]{Galdi}, it holds that
		\be
		\|\nabla \Bu \|_{L^4(\mathcal{B}_{\frac{5}{6}} (\Bx) \cap \Omega )} \leq C \| \Bu \|_{L^8(\mathcal{B}_{\frac{7}{8}} (\Bx) \cap \Omega )}^2 +C\|\Bu\|_{L^4(\mathcal{B}_{\frac{7}{8}} (\Bx) \cap \Omega)} \leq C, \ \  \text{for any}\, \Bx \in \partial \Omega.
		\ee
		Moreover, one has
		\be\label{bound-gradient-4}
		\| \nabla \Bu \|_{W^{1, 4} (\mathcal{B}_{\frac{3}{4}} (\Bx ) \cap \Omega )} \leq C \|\nabla \Bu \|_{L^4 (\mathcal{B}_{\frac{5}{6}} (\Bx) \cap \Omega ) }\|\Bu \|_{L^\infty (\mathcal{B}_{\frac{5}{6}} (\Bx) \cap \Omega ) } + C \|\Bu\|_{W^{1, 4} (\mathcal{B}_{\frac{5}{6}} (\Bx) \cap \Omega )   } \leq C .
		\ee

\be \label{bound-gradient-5}
\|\nabla^2 \Bu\|_{W^{1, 4} (\mathcal{B}_{\frac{1}{2}} (\Bx ) \cap \Omega ) } \leq C \|\Bu \cdot \nabla \Bu\|_{W^{1, 4} (\mathcal{B}_{\frac{3}{4}} (\Bx ) \cap \Omega )} + C \|\Bu\|_{W^{2, 4} (\mathcal{B}_{\frac{3}{4}} (\Bx) \cap \Omega )   } \leq C .
\ee
		Combining \eqref{bound-gradient-1}-\eqref{bound-gradient-4}  with Sobolev embedding inequality gives the uniform bound of $\nabla \Bu$, $\nabla^2 \Bu$, and hence the uniform bound of $\nabla P$.  The proof for Lemma \ref{bounded-gradient} is completed.
	\end{proof}

	\section{Axisymmetric solutions with no-slip boundary conditions}\label{Sec3}
	In this section, we deal with the axisymmtric solutions of the Navier-Stokes system in a slab with no-slip boundary conditions. The benefit of the no-slip boundary conditions is that Poincar\'{e} inequality
	\begin{equation}\label{Pinequality}
		\|\Bu\|_{L^2(\mathcal{O}_R)} \leq C \|\partial_z \Bu \|_{L^2(\mathcal{O}_R)}
	\end{equation}
	holds, where $C$ is a universal constant.  We will make use of this fact frequently. Furthermore, the axisymmetry of $\Bu$ also plays an important role, since it leads t improved estimates for the pressure term as indicated in \eqref{BetterEst}.
	\begin{proof}[{Proof of Theorem \ref{mainresult1}} ] The proof is divided into three steps.
		
		{\it Step 1. Set up}.  Assume that $\Bu$ is a smooth solution to \eqref{SNS} in $\Omega=\mathbb{R}^2\times (0,1)$ with no-slip boundary conditions. Multiplying the first equation in \eqref{SNS} by $\varphi_R(r) \Bu$ and integrating by parts one obtains
		\be \label{estimate1} \ba
		&\int_{\Omega} |\nabla \Bu|^2 \varphi_R =- \int_{\Omega} \nabla \varphi_R \cdot \nabla \Bu \cdot \Bu + \int_{\Omega} \frac12 |\Bu|^2\Bu \cdot \nabla \varphi_R  + \int_{\Omega}  P \Bu \cdot  \nabla \varphi_R   .
		\ea
		\ee
		Note that
		\be \label{proof-7}
		\int_{\Omega} P \Bu \cdot \nabla \varphi_R
		= 2\pi \int_0^1 \int_{R-1}^R P u^r r \, dr dz.
		\ee	
		Clearly, the last equation in \eqref{NScylindrical} for axisymmetric solutions can be reduced to
		\be \nonumber
		\partial_r ( r u^r)  + \partial_z ( ru^z ) = 0.
		\ee
		Hence for every fixed $r \geq 0$, one has
		\be
		\partial_r \int_0^1 r u^r \, dz  = - \int_0^1 \partial_z (r u^z) \, dz  = 0.
		\ee
		Thus it holds that
		\be
		\int_0^1 r u^r \, dz  = 0 \ \ \ \ \mbox{and}\ \ \ \ \ \int_0^1 \int_{R-1}^R ru^r \, dr dz = 0.
		\ee
		By virtue of Lemma \ref{Bogovskii}, there exists a vector valued function $\BP_R(r, z) \in H_0^1(D_R; \mathbb{R}^2)$ satisfying
		\be \label{proof-9}
		\partial_r \Psi_R^r + \partial_z \Psi_R^z =  r u^r
		\ee
		and
		\be \label{proof-10}
		\| \partial_r \BP_R \|_{L^2(D_R)} + \| \partial_z \BP_R \|_{L^2(D_R ) } \leq C \|r u^r \|_{L^2(D_R)} \leq C R^{\frac12}  \| u^r \|_{L^2(\mathcal{O}_R ) }.
		\ee

		Therefore, combining \eqref{proof-7} and \eqref{proof-9} one derives
		\be \label{proof-12}
		\int_{\Omega}  P \nabla\varphi_R\cdot  \Bu  =  \int_0^1 \int_{R-1}^R P \left( \partial_r \Psi_R^r + \partial_z \Psi_R^z  \right) \, drdz
		= - \int_0^1 \int_{R-1}^R (\partial_r P \Psi_R^r + \partial_z P \Psi_R^z ) \, drdz.
		\ee
		
		Since $\Bu$ is an axisymmetric solution of \eqref{NScylindrical}, the gradient  $(\partial_r P, \partial_z P)$ of the pressure satisfies
		\be \label{proof-15}
		\left\{
		\begin{aligned}
			&(u^r \partial_r + u^z \partial_z ) u^r - \frac{(u^\theta)^2}{ r} + \partial_r P = \left( \partial_r^2 + \frac1r \partial_r + \partial_z^2 - \frac{1}{r^2} \right) u^r,\\
			&(u^r \partial_r + u^z \partial_z ) u^z + \partial_z P = \left( \partial_r^2 + \frac1r \partial_r + \partial_z^2 \right) u^z.
		\end{aligned}
		\right.
		\ee
		Using  \eqref{proof-15} and integration by parts one obtains
		\be \label{proof-16}\ba
		& \int_0^1 \int_{R-1}^R \partial_r P \Psi_R^r \, drdz\\
		=&
		\int_0^1 \int_{R-1}^R \left( \partial_r^2 + \frac1r \partial_r + \partial_z^2 - \frac{1}{r^2} \right) u^r   \Psi_R^r \, drdz\\
		&\ \ \  - \int_0^1 \int_{R-1}^R \left[ (u^r \partial_r + u^z \partial_z )u^r - \frac{(u^\theta)^2}{r} \right]\Psi_R^r \, drdz \\
		=& - \int_0^1 \int_{R-1}^R ( \partial_r u^r \partial_r \Psi^r_R  + \partial_z u^r \partial_z \Psi_R^r  ) \, drdz + \int_0^1 \int_{R-1}^R \left( \frac1r \partial_r  - \frac{1}{r^2} \right) u^r   \Psi_R^r \, drdz  \\
		&\ \ \
		-\int_0^1 \int_{R-1}^R \left[ (u^r \partial_r + u^z \partial_z )u^r -
		\frac{(u^\theta)^2}{r} \right]\Psi_R^r \, drdz
		\ea \ee
		and
		\be \label{proof-16.5}\ba
		& \int_0^1 \int_{R-1}^R \partial_z P \Psi_R^z \, drdz\\
		=& - \int_0^1 \int_{R-1}^R ( \partial_r u^z \partial_r \Psi^z_R  + \partial_z u^z \partial_z \Psi_R^z  ) \, drdz + \int_0^1 \int_{R-1}^R \frac1r \partial_r u^r   \Psi_R^r \, drdz  \\
		&\ \ \
		-\int_0^1 \int_{R-1}^R  (u^r \partial_r + u^z \partial_z )u^z \Psi_R^z \, drdz
		\ea \ee
		
		{\it Step 2. Proof for Case (a) of Theorem \ref{mainresult1}.} Now we start to estimate the first two terms on the right hand side of \eqref{estimate1} and the terms on the right hand sides of \eqref{proof-16} and \eqref{proof-16.5}.
		By Poincar\'e inequality \eqref{Pinequality} and Sobolev embedding inequality, one derives
		\be \label{proof-3}
		\left| \int_{\Omega} \nabla \varphi_R \cdot \nabla \Bu \cdot \Bu  \right|
		\leq C \|\nabla \Bu \|_{L^2(\mathcal{O}_R)} \|\Bu\|_{L^2(\mathcal{O}_R ) }
		\leq C \| \nabla \Bu \|_{L^2(\mathcal{O}_R)}^2
		\ee
		and
		\be \label{proof-5} \ba
		\left| \int_{\Omega} \frac12 |\Bu|^2 \Bu \cdot \nabla \varphi_R  \right|
		\leq C R \int_0^1 \int_{R-1}^R |\Bu|^3 \, dr dz
		\leq \,& C R \| (\partial_r, \partial_z )\Bu \|_{L^2(D_R)}^3\\
		\leq\,& C R^{-\frac12} \|\nabla \Bu\|_{L^2(\mathcal{O}_R) }^3.
		\ea
		\ee

		By \eqref{proof-10} and Poincar\'e inequality, one has
		\be  \label{proof-17}
		\ba
		\left|  \int_0^1 \int_{R-1}^R ( \partial_r u^r \partial_r \Psi^r_R  + \partial_z u^r \partial_z \Psi_R^r  ) \, drdz \right|  \leq\, & C R^{-\frac12} \| \nabla \Bu\|_{L^2(\mathcal{O}_R) } \| \tilde{\nabla }\BP_R\|_{L^2(D_R)} \\
		\leq\, & C \| \nabla \Bu\|_{L^2(\mathcal{O}_R) }^2
		\ea \ee
		and
		\be \label{proof-18} \ba
		\left| \int_0^1 \int_{R-1}^R \left( \frac1r \partial_r  - \frac{1}{r^2} \right) u^r   \Psi_R^r \, drdz \right|
		\leq\,& C R^{-1} \| \tilde{\nabla} \Bu\|_{L^2(D_R)}  \|\BP_R\|_{L^2(D_R)}\\
		\leq\, & CR^{-1} \| \nabla \Bu\|_{L^2(\mathcal{O}_R) }^2.
		\ea
		\ee
		Furthermore, it holds that
		\be  \label{proof-19} \ba
		& \left| \int_0^1 \int_{R-1}^R \left[ (u^r \partial_r + u^z \partial_z )u^r - \frac{(u^\theta)^2}{r}  \right]\Psi_R^r \, drdz \right| \\[2mm]
		\leq\, & C \| (u^r, u^z, u^\theta ) \|_{L^4(D_R)}  \left( \| \tilde{\nabla } u^r \|_{L^2(D_R)} + R^{-1}  \|u^\theta \|_{L^2(D_R)} \right) \|\BP_R \|_{L^4(D_R)}\\[2mm]
		\leq\, & C \| (u^r, u^\theta, u^z ) \|_{L^2(D_R)}^{\frac12} \| (u^r, u^\theta, u^z) \|_{H^1(D_R)}^{\frac12} \cdot   R^{-\frac12} \|\nabla \Bu \|_{L^2(\mathcal{O}_R ) } \cdot  \|\tilde{\nabla } \BP_R\|_{L^2(D_R)} \\[2mm]
		\leq \,& C \| \tilde{\nabla } (u^r, u^\theta, u^z ) \|_{L^2(D_R)} \cdot   R^{-\frac12} \|\nabla \Bu \|_{L^2(\mathcal{O}_R ) } \cdot  R^{\frac12} \|u^r \|_{L^2(\mathcal{O}_R)}\\
		\leq \,& CR^{-\frac12} \|\nabla \Bu \|_{L^2(\mathcal{O}_R ) } ^3.
		\ea
		\ee
		Combining the estimates \eqref{proof-17}-\eqref{proof-19} one arrives at
		\be \label{proof-20}
		\left| \int_0^1 \int_{R-1}^R \partial_r P \Psi_R^r \, drdz \right| \leq C \|\nabla \Bu \|_{L^2(\mathcal{O}_R ) } ^2 + C R^{-\frac12} \| \nabla \Bu \|_{L^2(\mathcal{O}_R ) } ^3.
		\ee
		Similarly, one has
		\be \label{proof-21}
		\left| \int_0^1 \int_{R-1}^R \partial_z P \Psi_R^z \, drdz \right| \leq C \|\nabla \Bu \|_{L^2(\mathcal{O}_R ) } ^2 + C R^{-\frac12} \| \nabla \Bu \|_{L^2(\mathcal{O}_R ) } ^3.
		\ee
		Now it can be shown that
		\be \label{proof-26}
		\int_{\Omega} |\nabla \Bu|^2 \varphi_R \leq C \|\nabla \Bu \|_{L^2(\mathcal{O}_R ) } ^2 + C R^{-\frac12} \| \nabla \Bu \|_{L^2(\mathcal{O}_R ) } ^3.
		\ee
		
		Let
		\be \label{localenergy}
		Y(R) = \int_{0}^1\int_{\mathbb{R}^2} |\nabla \Bu(x)|^2 \varphi_R\left(\sqrt{x_1^2+x_2^2}\right) dx_1 dx_2 dx_3.
		\ee
		Note that for any three-dimensional function $\Bu$,  straightforward computations give
		\be \label{localenergy2}
		{Y}(R) =  \int_0^1 \int_0^{2\pi} \left( \int_0^{R-1} |\nabla \Bu|^2 r \, dr   +
		\int_{R-1}^R |\nabla \Bu|^2 (R-r) r \, dr \right) d\theta dz
		\ee
		and
		\be
		Y^{\prime} (R) = \int_{\mathcal{O}_R} |\nabla \Bu|^2 dx,
		\ee
		where the explicit form of $\varphi_R$ in \eqref{cut-off} has been used. For the axisymmetric solutions, $Y(R)$ in \eqref{localenergy2} can be reduced to
		\[
		{Y}(R) = 2\pi \int_0^1  \left( \int_0^{R-1} |\nabla \Bu|^2 r \, dr   +
		\int_{R-1}^R |\nabla \Bu|^2 (R-r) r \, dr \right) dz.
		\]
		Hence the estimate \eqref{proof-26} can be written as
		\[
		Y(R) \leq C Y^\prime (R) + C R^{-\frac12} \left( Y^{\prime} (R)\right)^{\frac32}.
		\]
		Then  Case (a)  of Theorem \ref{mainresult1}  follows from Lemma \ref{PL}.

		{\it Step 3.  Proof for Case (b) of Theorem \ref{mainresult1}.}\,
		We start from  \eqref{estimate1} and need to estimate the first two terms on the right hand side of \eqref{estimate1} and the terms on the right hand sides of \eqref{proof-16} and \eqref{proof-16.5} in a different way.
		Using  Poincar\'e inequality yields
		\be \label{coro1-2}
		\left|  \int_{\Omega} \nabla \varphi_R \cdot \nabla \Bu \cdot \Bu \right|
		\leq C \| \Bu \|_{L^2(\mathcal{O}_R)} \|\nabla \Bu \|_{L^2(\mathcal{O}_R)} \leq C R^{\frac12} \| \Bu \|_{L^\infty(\mathcal{O}_R)} \|\nabla \Bu \|_{L^2(\mathcal{O}_R)}
		\ee
		and
		\be \label{coro1-3} \ba
		\left| \int_{\Omega} \frac12 |\Bu|^2\Bu \cdot \nabla \varphi_R  \right|
		& \leq C R \| \Bu \|_{L^\infty(\mathcal{O}_R)}^2 \int_0^1 \int_{R-1}^R |(u^r, u^\theta, u^z) | \, dr dz \\
		& \leq C R  \| \Bu \|_{L^\infty(\mathcal{O}_R)}^2\left( \int_0^1 \int_{R-1}^R | (u^r, u^\theta, u^z) |^2 \, dr dz \right)^{\frac12}  \\
		& \leq C  R  \| \Bu \|_{L^\infty(\mathcal{O}_R)}^2  \| \partial_z \Bu \|_{L^2(D_R)} \leq C R^{\frac12}  \| \Bu \|_{L^\infty(\mathcal{O}_R)}^2 \| \nabla \Bu \|_{L^2(\mathcal{O}_R)}.
		\ea \ee
		
		According to \eqref{proof-10}, instead of \eqref{proof-17} and \eqref{proof-18}, one has
		\be \label{coro1-7} \ba
		& \left| \int_0^1 \int_{R-1}^R ( \partial_r u^r \partial_r \Psi^r_R  + \partial_z u^r \partial_z \Psi_R^r  ) \, drdz  \right|
		\leq C \| (\partial_r, \partial_z) u^r \|_{L^2(D_R)} \|  (\partial_r, \partial_z) \Psi_R^r \|_{L^2(D_R)} \\
		\leq\, & C R^{-\frac12}  \| \nabla \Bu \|_{L^2(\mathcal{O}_R) } \cdot  R^{\frac12} \| u^r \|_{L^2(\mathcal{O}_R)}  \leq C R^{\frac12} \| \Bu \|_{L^\infty(\mathcal{O}_R ) } \| \nabla \Bu \|_{L^2(\mathcal{O}_R)}
		\ea \ee
		and
		\be \label{coro1-8} \ba
		& \left|  \int_0^1 \int_{R-1}^R \left( \frac1r \partial_r  - \frac{1}{r^2} \right) u^r   \Psi_R^r \, drdz  \right| \\
		\leq\, & C R^{-1} \| \partial_r u^r \|_{L^2(D_R) }\| \Psi_R^r \|_{L^2(D_R) } + C R^{-2} \| u^r\|_{L^2(D_R)} \| \Psi_R^r \|_{L^2(D_R) } \\
		\leq\, & C R^{-\frac32} \| \nabla \Bu \|_{L^2(\mathcal{O}_R ) } \cdot R^{\frac12} \|u^r \|_{L^2(\mathcal{O}_R)} + CR^{-\frac52} \| \partial_z u^r \|_{L^2(\mathcal{O}_R)}  \cdot R^{\frac12} \|u^r \|_{L^2(\mathcal{O}_R)}\\
		\leq\, & C R^{-\frac12}\| \Bu \|_{L^\infty(\mathcal{O}_R ) } \| \nabla \Bu \|_{L^2(\mathcal{O}_R)} .
		\ea
		\ee
		Similarly, it holds that
		\be \label{coro1-9}\ba
		& \left| \int_0^1 \int_{R-1}^R \left[ (u^r \partial_r + u^z \partial_z )u^r -
		\frac{(u^\theta)^2}{r} \right]\Psi_R^r \, drdz \right| \\
		\leq\, & C \| \Bu \|_{L^\infty(\mathcal{O}_R ) } \| \tilde{\nabla } u^r  \|_{L^2(D_R)} \| \Psi_R^r \|_{L^2(D_R)} + CR^{-1} \| \Bu \|_{L^\infty(\mathcal{O}_R ) } \|u^\theta \|_{L^2(D_R)} \| \Psi_R^r \|_{L^2(D_R)} \\
		\leq\, & C R^{\frac12}\| \Bu \|_{L^\infty(\mathcal{O}_R ) }^2 \|\nabla \Bu \|_{L^2(\mathcal{O}_R ) } .
		\ea \ee
		Combining \eqref{coro1-7}-\eqref{coro1-9} one arrives at
		\be \label{coro1-10}
		\left| \int_0^1 \int_{R-1}^R \partial_r P \Psi_R^r \, drdz \right| \leq C R^{\frac12}\left( \| \Bu \|_{L^\infty(\mathcal{O}_R ) } + \| \Bu \|_{L^\infty(\mathcal{O}_R ) }^2 \right) \| \nabla \Bu \|_{L^2(\mathcal{O}_R ) } .
		\ee
		The similar computations show that
		\be \label{coro1-11}
		\left| \int_0^1 \int_{R-1}^R \partial_z P \Psi_R^z \, drdz \right| \leq C R^{\frac12}\left( \| \Bu \|_{L^\infty(\mathcal{O}_R ) } + \| \Bu \|_{L^\infty(\mathcal{O}_R ) }^2 \right)   \| \nabla \Bu \|_{L^2(\mathcal{O}_R ) } .
		\ee
		Therefore, one has
		\be \label{coro1-12}
		Y(R) \leq C R^{\frac12}\left( \| \Bu \|_{L^\infty(\mathcal{O}_R ) } + \| \Bu \|_{L^\infty(\mathcal{O}_R ) }^2 \right)  \left[ Y^{\prime}(R) \right]^{\frac12},
		\ee
		where $Y(R)$ is defined in \eqref{localenergy}.
		
		Suppose $\Bu$ is not identically equal to zero and $\Bu$ satisfies \eqref{growth-coro1}. For any small $\epsilon >0$, there exists a $R_0(\epsilon) >2$ such that
		\be \nonumber
		\| \Bu \|_{L^\infty(\mathcal{O}_R ) } \leq \epsilon R \ \ \ \ \ \ \text{for any}\,\,  R\geq R_0(\epsilon).
		\ee
		Hence the inequality  \eqref{coro1-12} implies that
		\be \label{coro1-14}
		\frac{Y^{\prime}(R) }{ [Y(R)]^2} \geq (C\epsilon)^{-2} R^{-5}.
		\ee
	If $\Bu$ is not equal to zero, according to Case (a) of Theorem \ref{mainresult1}, $Y(R)$ must be unbounded as $R\rightarrow + \infty$.
		For every $R$ sufficiently large, integrating \eqref{coro1-14} over $[R, + \infty)$ one arrives at
		\be \label{coro1-15}
		\frac{1}{Y(R)}  \geq  \frac14 (C \epsilon)^{-2} R^{-4}.
		\ee
Since $\epsilon$ can be arbitrarily small,	this implies  \eqref{growth} and leads to a contradiction with the assumption that $\Bu$ is not identically zero.    Hence the proof of Theorem \ref{mainresult1} is completed.
	\end{proof}



	\section{General 3D solutions in a slab with no-slip boundary conditions}\label{Sec4}
	This section is devoted to the study for general solutions of the Navier-Stokes system \eqref{SNS} in a slab with no-slip boundary conditions. Since $\partial_\theta P$ does not have the same scaling as $\partial_r P$ and $\partial_z P$, this makes $\partial_\theta P$ a  troublesome term for the estimate,  whereas it does not appear in the axisymmetric setting. The key ideas to deal with other terms in general case are almost the same as that for the axisymmetric case.
	\begin{proof}[{Proof of Theorem \ref{mainresult4}}] The proof contains three steps.
		
		{\it  Step 1. Set up.}\ Assume that $\Bu$ is a smooth solution to the Navier-Stokes system \eqref{SNS}.
		The equality  \eqref{estimate1} still holds.
			Instead of \eqref{proof-7}, one has
		\be \label{proof4-7}
		\int_{\Omega} P \Bu \cdot \nabla \varphi_R
		=  \int_0^1  \int_0^{2\pi} \int_{R-1}^R P u^r r \, dr  d\theta dz .
		\ee		
			It follows from the last equation in \eqref{NScylindrical} (divergence free equation) that
	 for every fixed $r\geq 0$, one has
		\be \label{proof4-10}
		\partial_r \int_0^1 \int_0^{2\pi } ru^r \, d\theta dz = - \int_0^1 \int_0^{2\pi} \partial_\theta u^\theta \, d\theta dz - \int_0^1 \int_0^{2\pi}
		\partial_z (r u^z) \, d\theta dz = 0.
		\ee
		And then it holds that
		\be \label{proof4-12}
		\int_0^1 \int_0^{2\pi} r u^r \, d\theta dz = 0\ \ \ \  \text{and}\quad \int_0^1 \int_0^{2\pi} \int_{R-1}^R r u^r \, dr d\theta  dz =0.
		\ee
		By virtue of Lemma \ref{Bogovskii}, there exists a vector valued function $\BP_R(r, \theta, z)\in H_0^1(\mathcal{D}_R; \mathbb{R}^3) $ satisfying
		\be \label{proof4-13}
		\partial_r \Psi_R^r + \partial_\theta \Psi_R^\theta + \partial_z \Psi_R^z  = ru^r\ \ \ \ \mbox{in}\ \mathcal{D}_R
		\ee
		and
		\be \label{proof4-15}\ba
		&\|\partial_r \BP_R\|_{L^2(\mathcal{D}_R ) } + \|\partial_\theta \BP_R\|_{L^2(\mathcal{D}_R)} + \|\partial_z \BP_R\|_{L^2(\mathcal{D}_R)}
		\leq C  \|r u^r \|_{L^2(\mathcal{D}_R)} \leq C  R^{\frac12} \| u^r \|_{L^2(\mathcal{O}_R)}.
		\ea \ee
		Combining \eqref{proof4-7} and \eqref{proof4-13} one obtains
		\be \label{proof4-16} \ba
		\int_{\Omega}  P \Bu \cdot \nabla \varphi_R  &  =    \int_0^1 \int_0^{2\pi} \int_{R-1}^R P \left( \partial_r \Psi_R^r +  \partial_\theta \Psi_R^\theta + \partial_z \Psi_R^z  \right) \, dr d\theta dz \\
		& = - \int_0^1 \int_0^{2\pi} \int_{R-1}^R (\partial_r P \Psi_R^r + \partial_\theta P \Psi_R^\theta +  \partial_z P \Psi_R^z ) \, dr d\theta dz.
		\ea \ee
		Furthermore, it follows from the momentum equations in \eqref{NScylindrical} that one has
		\be \label{proof4-20}
		\ba
		&\int_0^1 \int_0^{2\pi} \int_{R-1}^R \partial_r P \Psi_R^r \, dr d\theta dz \\
		=& - \int_0^1 \int_0^{2\pi} \int_{R-1}^R \left(\partial_r u^r \partial_r \Psi_R^r + \partial_z u^r \partial_z \Psi_R^r +  \frac{1}{r^2} \partial_\theta u^r \partial_\theta \Psi_R^r \right)  \, drd\theta dz  \\
		&+ \int_0^1 \int_0^{2\pi} \int_{R-1}^R \left[\left(  \frac{1}{r} \partial_r - \frac{1}{r^2}    \right) u^r +   \frac{(u^\theta)^2}{r} - \frac{2}{r^2} \partial_\theta u^\theta \right]\Psi_R^r \, dr d\theta dz \\
		& - \int_0^1 \int_0^{2\pi} \int_{R-1}^R \left(u^r \partial_r + \frac{u^\theta}{r}  \partial_\theta  + u^z \partial_z \right) u^r \Psi_R^r \, dr d\theta  dz ,
		\ea
		\ee
\be \label{proof4-30}
\ba
&\int_0^1 \int_0^{2\pi} \int_{R-1}^R \partial_\theta P \Psi_R^\theta  \, dr d\theta dz \\
=\, & - \int_0^1 \int_0^{2\pi} \int_{R-1}^R  \left[ r(\partial_r u^\theta \partial_r \Psi_R^\theta + \partial_z u^\theta \partial_z \Psi_R^\theta) +  r^{-1} \partial_\theta u^\theta \partial_\theta \Psi_R^\theta \right] \, drd\theta dz  \\
& + \int_0^1 \int_0^{2\pi} \int_{R-1}^R \left[- \frac{1}{r}   u^\theta - u^\theta u^r  + \frac{2}{r} \partial_\theta u^r \right]\Psi_R^\theta  \, dr d\theta dz \\
& - \int_0^1 \int_0^{2\pi} \int_{R-1}^R r \left(u^r \partial_r + \frac{1}{r} u^\theta \partial_\theta  + u^z \partial_z \right) u^\theta \Psi_R^\theta  \, dr d\theta  dz ,
\ea
\ee	
and	
	\be \label{proof4-20.5}
\ba
&\int_0^1 \int_0^{2\pi} \int_{R-1}^R \partial_z P \Psi_R^z \, dr d\theta dz \\
=& - \int_0^1 \int_0^{2\pi} \int_{R-1}^R \left(\partial_r u^z \partial_r \Psi_R^z + \partial_z u^z \partial_z \Psi_R^z +  \frac{1}{r^2} \partial_\theta u^z \partial_\theta \Psi_R^z \right)  \, drd\theta dz  \\
& - \int_0^1 \int_0^{2\pi} \int_{R-1}^R \left(u^r \partial_r + \frac{u^\theta}{r}  \partial_\theta  + u^z \partial_z -\frac{1}{r}\partial_r \right) u^z \Psi_R^z \, dr d\theta  dz .
\ea
\ee
		
{\it Step 2. Proof for Case (a) of Theorem \ref{mainresult4}.}	Now we estimate  the first two terms on the right hand side of \eqref{estimate1} and the terms on the right hand sides of \eqref{proof4-20}--\eqref{proof4-20.5}. First,  Poincar\'e inequality and Sobolev embedding inequality lead to
	\be \label{proof4-4.5}
\left|  \int_{\Omega} \nabla \varphi_R \cdot \nabla \Bu \cdot \Bu \right|
\leq  \| \Bu \|_{L^2(\mathcal{O}_R)} \|\nabla \Bu \|_{L^2(\mathcal{O}_R)} \leq    \|\nabla \Bu \|^2_{L^2(\mathcal{O}_R)}
\ee
and
		\be \label{proof4-5} \ba
		\left| \int_{\Omega} \frac12 |\Bu|^2 \Bu \cdot \nabla \varphi_R  \right|
		&\leq C R \int_0^1 \int_0^{2\pi}  \int_{R-1}^R  |\Bu|^3 \, dr d\theta  dz  \\
		& \leq C R  \left\| (u^r, u^\theta, u^z) \right\|_{L^2(\mathcal{D}_R ) }^{\frac32}  \left\| \bar{\nabla} ( u^r, u^\theta, u^z ) \right\|_{L^2(\mathcal{D}_R )}^{\frac32} \\
		& \leq C R\cdot  R^{-\frac34} \|\Bu\|_{L^2(\mathcal{O}_R )}^{\frac32}\cdot R^{\frac34} \| \nabla \Bu\|_{L^2(\mathcal{O}_R)}^{\frac32} \\
		& \leq C R \|\nabla \Bu\|_{L^2(\mathcal{O}_R )}^3.
		\ea
		\ee
	
		Using \eqref{proof4-15} one obtains
		\be \label{proof4-21}
		\ba
		& \left| \int_0^1 \int_0^{2\pi} \int_{R-1}^R \left(\partial_r u^r \partial_r \Psi_R^r + \partial_z u^r \partial_z \Psi_R^r +  \frac{1}{r^2} \partial_\theta u^r \partial_\theta \Psi_R^r \right) \, drd\theta dz \right| \\
		\leq\, & C \| ( \partial_r, r^{-1} \partial_\theta, \partial_z ) u^r \|_{L^2(\mathcal{D}_R)} \| ( \partial_r,  \partial_\theta, \partial_z ) \Psi_R^r\|_{L^2(\mathcal{D}_R)}\\
		\leq\, & C R^{-\frac12} \| \nabla \Bu \|_{L^2(\mathcal{O}_R)} \cdot R^{\frac12} \| u^r \|_{L^2(\mathcal{O}_R)} \\
		\leq\, & C  \|\nabla \Bu \|_{L^2(\mathcal{O}_R)}^2
		\ea
		\ee
		and
		\be \label{proof4-22} \ba
		& \left|  \int_0^1 \int_0^{2\pi} \int_{R-1}^R \left[ \left( \frac{1}{r} \partial_r   - \frac{1}{r^2}       \right)u^r + \frac{2}{r^2}\partial_\theta u^\theta \right]  \Psi_R^r  \, dr d\theta dz      \right| \\
		\leq\, &   C \left(  R^{-1} \| \partial_r u^r \|_{L^2(\mathcal{D}_R)} + R^{-2} \| u^r \|_{L^2(\mathcal{D}_R)} + R^{-1} \|r^{-1} \partial_\theta u^\theta \|_{L^2(\mathcal{D}_R)}\right) \| \Psi_R^r \|_{L^2(\mathcal{D}_R)} \\
		\leq\, & C ( R^{-\frac32} + R^{-\frac52}   ) \|\nabla \Bu\|_{L^2(\mathcal{O}_R)}  \cdot   \| \partial_z \Psi_R^r \|_{L^2(\mathcal{D}_R)} \\
		\leq\, & C ( R^{-\frac32} + R^{-\frac52}  ) \|\nabla \Bu\|_{L^2(\mathcal{O}_R)}  \cdot   R^{\frac12} \|u^r \|_{L^2(\mathcal{O}_R)} \\
		\leq\, & C R^{-1}    \| \nabla \Bu\|_{L^2(\mathcal{O}_R ) }^2.
		\ea \ee
		 By  H\"{o}lder inequality it can be shown that
		\be \label{proof4-23}
		\ba
		&  \left| \int_0^1 \int_0^{2\pi} \int_{R-1}^R \left(u^r \partial_r + \frac{1}{r} u^\theta \partial_\theta  + u^z \partial_z \right) u^r \Psi_R^r \, dr d\theta  dz \right| \\
		\leq\, & C  \| (u^r, u^\theta, u^z) \|_{L^3(\mathcal{D}_R) } \| (\partial_r, r^{-1} \partial_\theta , \partial_z )u^r \|_{L^2(\mathcal{D}_R)}
		\| \Psi_R^r \|_{L^6(\mathcal{D}_R)}.
		\ea
		\ee
		It follows from Gagliardo-Nirenberg inequality and Poincar\'{e} inequality that
		\be \label{proof4-24}
		\ba
		& \| (u^r, u^\theta, u^z ) \|_{L^3(\mathcal{D}_R )} \\
		\leq\, & C \| (u^r, u^\theta, u^z ) \|_{L^2(\mathcal{D}_R)}^{\frac12} \| (u^r, u^\theta, u^z ) \|_{H^1(\mathcal{D}_R)}^{\frac12} \\
		\leq\,& C   \| \partial_z (u^r, u^\theta, u^z ) \|_{L^2(\mathcal{D}_R)}^{\frac12} \left(  \| \partial_z (u^r, u^\theta, u^z)  \|_{L^2(\mathcal{D}_R)}
		+ \| \bar{\nabla } (u^r, u^\theta, u^z ) \|_{L^2(\mathcal{D}_R)}      \right)^{\frac12}\\
		\leq\, & C  R^{-\frac14} \| \nabla \Bu \|_{L^2(\mathcal{O}_R)}^{\frac12} \left(   R^{-\frac12} \| \nabla \Bu \|_{L^2(\mathcal{O}_R)} +
		R^{\frac12}\| \nabla \Bu \|_{L^2(\mathcal{O}_R)}   \right)^{\frac12} \\
		\leq\, & C    \| \nabla \Bu \|_{L^2(\mathcal{O}_R)}
		\ea
		\ee
		and
		\be \label{proof4-25}
		\| \BP_R\|_{L^6(\mathcal{D}_R)} \leq C \|\bar{\nabla } \BP_R \|_{L^2(\mathcal{D}_R ) } \leq C R^{\frac12}\|u^r \|_{L^2(\mathcal{O}_R )}
		\leq C R^{\frac12} \|\nabla \Bu \|_{L^2(\mathcal{O}_R )}.
		\ee
		Putting \eqref{proof4-24}-\eqref{proof4-25} into \eqref{proof4-23} one derives
		\be \label{proof4-25-1}
		\left| \int_0^1 \int_0^{2\pi} \int_{R-1}^R \left(u^r \partial_r + \frac{1}{r} u^\theta \partial_\theta  + u^z \partial_z \right) u^r \Psi_R^r \, dr d\theta  dz \right| \leq C   \| \nabla \Bu \|_{L^2(\mathcal{O}_R )}^3.
		\ee
		Similarly, it follows from \eqref{proof4-24}-\eqref{proof4-25} that
		\begin{align}
			\begin{aligned}
				\label{proof4-26}
				&\left| \int_0^1 \int_0^{2\pi} \int_{R-1}^R \frac{(u^\theta)^2}{r}  \Psi_R^r \, dr d\theta  dz \right|
				\leq CR^{-1} \|u^\theta \|_{L^3(\mathcal{D}_R)} \|u^\theta \|_{L^2(\mathcal{D}_R)} \| \Psi_R^r \|_{L^6 (\mathcal{D}_R )}
				\\
				\leq\, & C R^{-1} \|\nabla \Bu \|_{L^2(\mathcal{O}_R)} \cdot \|\partial _z u^\theta\|_{L^2(\mathcal{D}_R)} \cdot R^{\frac12} \|\nabla \Bu \|_{L^2(\mathcal{O}_R)}
				\\
				\leq\, & C R^{-1} \|\nabla \Bu \|_{L^2(\mathcal{O}_R)} \cdot R^{-\frac12}\|\nabla \Bu\|_{L^2(\mathcal{O}_R)} \cdot R^{\frac12} \|\nabla \Bu \|_{L^2(\mathcal{O}_R)}
				\\
				\leq\, & R^{-1} \|\nabla \Bu \|_{L^2(\mathcal{O}_R)}^3.
			\end{aligned}
		\end{align}
		Combining the estimates \eqref{proof4-21}-\eqref{proof4-26} one obtains
		\be \label{proof4-27}  \left| \int_0^1 \int_0^{2\pi} \int_{R-1}^R \partial_r P \Psi_R^r \, dr d\theta dz \right|
		\leq C  \| \nabla \Bu \|_{L^2(\mathcal{O}_R )}^2 + C  \| \nabla \Bu\|_{L^2(\mathcal{O}_R)}^3.
		\ee

		By making use of the estimates \eqref{proof4-15} and  \eqref{proof4-24}-\eqref{proof4-25}  it can be shown that
		\be \label{proof4-31}
		\ba
		&  \int_0^1 \int_0^{2\pi} \int_{R-1}^R \left[ r(\partial_r u^\theta \partial_r \Psi_R^\theta + \partial_z u^\theta \partial_z \Psi_R^\theta) +  r^{-1} \partial_\theta u^\theta \partial_\theta \Psi_R^\theta \right] \, drd\theta dz \\
		\leq\, & C R \| (\partial_r u^\theta, \partial_z u^\theta, r^{-2} \partial_\theta  u^\theta )\|_{L^2(\mathcal{D}_R)} \| \bar{\nabla} \Psi_R^\theta \|_{L^2(\mathcal{D}_R)} \\
		\leq\, & C R^{\frac12} \| \nabla \Bu \|_{L^2(\mathcal{O}_R)}\cdot R^{\frac12} \| u^r \|_{L^2(\mathcal{O}_R) }  \\
		\leq\, & C R \|\nabla \Bu \|_{L^2(\mathcal{O}_R)}^2
		\ea
		\ee
		and
		\be \label{proof4-32}
		\ba
		& \left| \int_0^1 \int_0^{2\pi} \int_{R-1}^R \left[ - \frac{1}{r}     u^\theta  + \frac{2}{r} \partial_\theta u^r \right]\Psi_R^\theta  \, dr d\theta dz \right|\\
		\leq\, & C  \left(  R^{-1} \| u^\theta \|_{L^2(\mathcal{D}_R)} + \|r^{-1} \partial_\theta u^r \|_{L^2(\mathcal{D}_R)} \right) \| \Psi_R^\theta \|_{L^2(\mathcal{D}_R)}\\
		\leq\, & C \left( R^{-\frac32 } \|u^\theta \|_{L^2 (\mathcal{O}_R)} + R^{-\frac12} \| \nabla \Bu \|_{L^2(\mathcal{O}_R ) } \right)  \cdot  R^{\frac12} \| u^r \|_{L^2(\mathcal{O}_R) }\\
		\leq\, & C  \|\nabla \Bu \|_{L^2(\mathcal{O}_R )}^2 .
		\ea
		\ee
		Furthermore, one has
		\be \label{proof4-33} \ba
		\left| \int_0^1 \int_0^{2\pi} \int_{R-1}^R  u^r u^{\theta} \Psi_R^{\theta}   \, dr d\theta dz \right|
		\leq & C \|u^r \|_{L^3(\mathcal{D}_R)} \|u^\theta \|_{L^2(\mathcal{D}_R)} \| \Psi_R^\theta \|_{L^6(\mathcal{D}_R)} \\
		\leq & C   \| \nabla \Bu\|_{L^2(\mathcal{O}_R)} \cdot  R^{-\frac12} \| \nabla \Bu \|_{L^2(\mathcal{O}_R)} \cdot R^{\frac12} \| u^r \|_{L^2(\mathcal{O}_R)}\\
		\leq & C  \| \nabla \Bu\|_{L^2(\mathcal{O}_R)}^3
		\ea  \ee
		and
		\be  \label{proof4-35} \ba
		& \left| \int_0^1 \int_0^{2\pi} \int_{R-1}^R \left(r u^r \partial_r +  u^\theta \partial_\theta  + r u^z \partial_z \right) u^\theta \Psi_R^\theta  \, dr d\theta  dz \right| \\
		\leq\, & C R \| (u^r, u^\theta, u^z) \|_{L^3(\mathcal{D}_R)} \| (\partial_r, r^{-1} \partial_\theta, \partial_z )u^\theta \|_{L^2(\mathcal{D}_R)} \| \Psi_R^\theta \|_{L^6(\mathcal{D}_R )} \\
		\leq\, & C R  \| \nabla \Bu \|_{L^2(\mathcal{O}_R)}^3.
		\ea
		\ee
		Combining the estimates \eqref{proof4-31}-\eqref{proof4-35} one arrives at
		\be \label{proof4-38}
		\left| \int_0^1 \int_0^{2\pi} \int_{R-1}^R \partial_\theta P \Psi_R^\theta  \, dr d\theta dz \right|
		\leq C R \| \nabla \Bu \|_{L^2(\mathcal{O}_R)}^2 + C R   \| \nabla \Bu\|_{L^2(\mathcal{O}_R)}^3.  \ee
		
		Similarly, it can be proved that
		\be \label{proof4-50}  \left| \int_0^1 \int_0^{2\pi} \int_{R-1}^R \partial_z P \Psi_R^z  \, dr d\theta dz \right|
		\leq C\| \nabla \Bu \|_{L^2(\mathcal{O}_R )}^2 + C  \| \nabla \Bu\|_{L^2(\mathcal{O}_R)}^3.
		\ee

		The above computations imply
		\be \label{proof4-51}
		{Y}(R) \leq C_1 R {Y}^{\prime}(R) + C_2 R \left[ {Y}^{\prime}(R) \right]^{\frac32}.
		\ee
		It follows from Case (c) of Lemma \ref{PL} that if $Y(R)$ is not identically zero, then for any $\alpha\in (0, \frac{1}{C_1})$, there exists a constant $C>0$ such that
		\[
		Y(R)\geq CR^\alpha.
		\]
		Hence if $\varliminf _{R\to\infty}(Y(R) R^{-\alpha})= 0$, then $Y(R)$ must be identically zero. This implies that $\nabla \Bu\equiv 0$ and thus $\Bu\equiv 0$. Hence the proof for Case (a) of Theorem  \ref{mainresult4} is completed.

		{\it Step 3.  Proof for Case (b) of Theorem  \ref{mainresult4}.}\,   We estimate  the first two terms on the right hand side of \eqref{estimate1} and the terms on the right hand sides of \eqref{proof4-20}--\eqref{proof4-20.5} in a different way.
		Using Poincar\'e inequality and Sobolev embedding inequality we obtain
		\be \label{proof5-1}
		\left| \int_{\Omega} \nabla \varphi_R \cdot \nabla \Bu \cdot \Bu  \right|
		\leq C \|\nabla \Bu \|_{L^2(\mathcal{O}_R)} \|\Bu\|_{L^2(\mathcal{O}_R ) }
		\leq C  R^{\frac12} \|\Bu\|_{L^\infty (\mathcal{O}_R)} \| \nabla \Bu \|_{L^2(\mathcal{O}_R)}
		\ee
		and
		\be \label{proof5-5}
		\begin{aligned}
			 \left| \int_{\Omega}  \frac12 |\Bu|^2 \Bu \cdot \nabla \varphi_R \right|
		\leq \,& C \|u^r\|_{L^\infty (\mathcal{O}_R )}
		\|\Bu\|_{L^2 (\mathcal{O}_R )}^2 \\
		\leq \, & CR^{\frac12}\|u^r\|_{L^\infty (\mathcal{O}_R )}
		\|\Bu\|_{L^\infty (\mathcal{O}_R )}
		\|\nabla \Bu\|_{L^2 (\mathcal{O}_R )}.
	\end{aligned}
		\ee
	
		By virtue of \eqref{proof4-15}, one has
		\be \label{proof5-21}
		\ba
		& \left| \int_0^1 \int_0^{2\pi} \int_{R-1}^R \left(\partial_r u^r \partial_r \Psi_R^r + \partial_z u^r \partial_z \Psi_R^r +  \frac{1}{r^2} \partial_\theta u^r \partial_\theta \Psi_R^r \right) \, drd\theta dz \right| \\
		\leq\, & C \| ( \partial_r, r^{-1} \partial_\theta, \partial_z ) u^r \|_{L^2(\mathcal{D}_R)} \| ( \partial_r,  \partial_\theta, \partial_z ) \Psi_R^r\|_{L^2(\mathcal{D}_R)}\\
		\leq\, & C  \| \nabla \Bu \|_{L^2(\mathcal{O}_R)}  \| u^r \|_{L^2(\mathcal{O}_R)} \leq C R^{\frac12} \|u^r\|_{L^\infty (\mathcal{O}_R)}\| \nabla \Bu \|_{L^2(\mathcal{O}_R)}
		\ea
		\ee
		and
		\be \label{proof5-22} \ba
		& \left|  \int_0^1 \int_0^{2\pi} \int_{R-1}^R \left[ \left( \frac{1}{r} \partial_r   - \frac{1}{r^2}       \right)u^r + \frac{2}{r^2}\partial_\theta u^\theta \right]  \Psi_R^r  \, dr d\theta dz      \right| \\
		\leq\, &   C \left(  R^{-1} \| \partial_r u^r \|_{L^2(\mathcal{D}_R)} + R^{-2} \| u^r \|_{L^2(\mathcal{D}_R)} + R^{-1} \|r^{-1} \partial_\theta u^\theta \|_{L^2(\mathcal{D}_R)}\right) \| \Psi_R^r \|_{L^2(\mathcal{D}_R)} \\
		\leq\, & C ( R^{-\frac32} + R^{-\frac52}   ) \|\nabla \Bu\|_{L^2(\mathcal{O}_R)}  \cdot   \| \partial_z \Psi_R^r \|_{L^2(\mathcal{D}_R)} \\
		\leq\, & C R^{-1}    \| \nabla \Bu\|_{L^2(\mathcal{O}_R ) } \|u^r\|_{L^2(\mathcal{O}_R )}
		\leq C R^{-\frac12} \|u^r\|_{L^\infty (\mathcal{O}_R)} \|\nabla \Bu \|_{L^2(\mathcal{O}_R)} .
		\ea \ee
		Furthermore, it holds that
		\be \label{proof5-23}
		\ba
		&  \left| \int_0^1 \int_0^{2\pi} \int_{R-1}^R \left(u^r \partial_r + \frac{1}{r} u^\theta \partial_\theta  + u^z \partial_z \right) u^r \Psi_R^r \, dr d\theta  dz \right| \\
		\leq\, & C  \| (u^r, u^\theta, u^z) \|_{L^\infty (\mathcal{D}_R) } \| (\partial_r, r^{-1} \partial_\theta , \partial_z )u^r \|_{L^2(\mathcal{D}_R)}
		\| \Psi_R^r \|_{L^2 (\mathcal{D}_R)}\\
		\leq\, &  C \|\Bu\|_{L^\infty (\mathcal{O}_R)} \| \nabla \Bu \|_{L^2(\mathcal{O}_R)}  \|u^r \|_{L^2(\mathcal{O}_R)}
		\leq C R^{\frac12} \|\Bu\|_{L^\infty (\mathcal{O}_R)}
		\|u^r\|_{L^\infty (\mathcal{O}_R)}
		\| \nabla \Bu \|_{L^2(\mathcal{O}_R)}
		\ea
		\ee
		and
		\be \label{proof5-27}
		\ba
		\left| \int_0^1 \int_0^{2\pi} \int_{R-1}^R \frac{(u^\theta)^2}{r}  \Psi_R^r \, dr d\theta  dz \right|
		&\leq CR^{-1} \|\Bu\|_{L^\infty (\mathcal{O}_R)}  \|u^\theta \|_{L^2(\mathcal{D}_R) } \| \Psi_R^r \|_{L^2(\mathcal{D}_R )}
		\\
		& \leq CR^{-1}\|\Bu\|_{L^\infty (\mathcal{O}_R)}  \|\nabla \Bu \|_{L^2(\mathcal{O}_R )} \| u^r \|_{L^2(\mathcal{O}_R )}
		\\
		&\leq C R^{-\frac12} \|\Bu\|_{L^\infty (\mathcal{O}_R)}
		\| u^r \|_{L^\infty(\mathcal{O}_R )}
		\| \nabla \Bu\|_{L^2(\mathcal{O}_R ) }.
		\ea
		\ee
		Combining the estimates \eqref{proof5-21}-\eqref{proof5-27} one arrives at
		\be \label{proof5-28}   \left| \int_0^1 \int_0^{2\pi} \int_{R-1}^R \partial_r P \Psi_R^r \, dr d\theta dz \right|
		\leq
		C R^{\frac12}(1+\|\Bu\|_{L^\infty (\mathcal{O}_R)})\|u^r\|_{L^\infty (\mathcal{O}_R)} \|\nabla \Bu \|_{L^2(\mathcal{O}_R )} .
		\ee

	Furthermore,  straightforward computations lead to
		\be \label{proof5-31}
		\ba
		&  \int_0^1 \int_0^{2\pi} \int_{R-1}^R \left[ r(\partial_r u^\theta \partial_r \Psi_R^\theta + \partial_z u^\theta \partial_z \Psi_R^\theta) +  r^{-1} \partial_\theta u^\theta \partial_\theta \Psi_R^\theta \right] \, drd\theta dz \\
		\leq\, & C R \| (\partial_r u^\theta, \partial_z u^\theta, r^{-2} \partial_\theta  u^\theta) \|_{L^2(\mathcal{D}_R)} \| \bar{\nabla} \Psi_R^\theta \|_{L^2(\mathcal{D}_R)} \\
		\leq\, & C R  \| \nabla \Bu \|_{L^2(\mathcal{O}_R)}  \| u^r \|_{L^2(\mathcal{O}_R) }
		\leq C R^{\frac32} \|u^r\|_{L^\infty(\mathcal{O}_R)}\| \nabla \Bu \|_{L^2(\mathcal{O}_R)}
		\ea
		\ee
		and
		\be \label{proof5-32}
		\ba
		& \left| \int_0^1 \int_0^{2\pi} \int_{R-1}^R \left(  - \frac{1}{r}    u^\theta  + \frac{2}{r} \partial_\theta u^r \right)\Psi_R^\theta  \, dr d\theta dz \right|\\
		\leq\, & C  \left( R^{-1} \| u^\theta \|_{L^2(\mathcal{D}_R)} + \|r^{-1} \partial_\theta u^r \|_{L^2(\mathcal{D}_R)} \right) \| \Psi_R^\theta \|_{L^2(\mathcal{D}_R)}\\
		\leq\, & C   \|\nabla \Bu \|_{L^2(\mathcal{O}_R )} \|u^r \|_{L^2 (\mathcal{O}_R)} \leq C R^{\frac32} \|u^r\|_{L^\infty(\mathcal{O}_R)} \| \nabla \Bu \|_{L^2(\mathcal{O}_R)} .
		\ea
		\ee
		By H\"{o}lder inequality and \eqref{proof4-15} one derives
		\be \label{proof5-33} \ba
		\left| \int_0^1 \int_0^{2\pi} \int_{R-1}^R  u^r u^{\theta} \Psi_R^{\theta}   \, dr d\theta dz \right|
		\leq\, & C \|u^r \|_{L^\infty (\mathcal{D}_R ) }  \|u^\theta \|_{L^2(\mathcal{D}_R)} \| \Psi_R^\theta \|_{L^2(\mathcal{D}_R)} \\
		\leq\, & C \|u^r\|_{L^\infty(\mathcal{O}_R)} \| \nabla \Bu\|_{L^2(\mathcal{O}_R)}  \| u^r \|_{L^2(\mathcal{O}_R)}
		\\
		\leq& C R^{\frac12} \|u^r\|_{L^\infty (\mathcal{O}_R)}^2  \| \nabla \Bu \|_{L^2(\mathcal{O}_R)}
		\ea  \ee
		and
		\be  \label{proof5-35} \ba
		& \left| \int_0^1 \int_0^{2\pi} \int_{R-1}^R \left(r u^r \partial_r +  u^\theta \partial_\theta  + r u^z \partial_z \right) u^\theta \Psi_R^\theta  \, dr d\theta  dz \right| \\
		\leq\, & C R \| \Bu \|_{L^\infty (\mathcal{O}_R)} \| (\partial_r, r^{-1} \partial_\theta, \partial_z )u^\theta \|_{L^2(\mathcal{D}_R)} \| \Psi_R^\theta \|_{L^2 (\mathcal{D}_R )} \\
		\leq\, & C R\| \Bu \|_{L^\infty (\mathcal{O}_R)}     \| \nabla \Bu \|_{L^2(\mathcal{O}_R)}  \| u^r \|_{L^2(\mathcal{O}_R )}
		\\
		\leq\, & C R^{\frac32} \| \Bu \|_{L^\infty (\mathcal{O}_R)}\| u^r \|_{L^\infty (\mathcal{O}_R)}  \| \nabla \Bu \|_{L^2(\mathcal{O}_R)}.
		\ea
		\ee
		Combining the estimates \eqref{proof5-31}-\eqref{proof5-35} one obtains
		\be \label{proof5-38}
		\begin{aligned}
		&\left| \int_0^1 \int_0^{2\pi} \int_{R-1}^R \partial_\theta P \Psi_R^\theta  \, dr d\theta dz \right|
		\leq   C R^{\frac32}
		(1+ \| \Bu \|_{L^\infty (\mathcal{O}_R)})
		\| u^r \|_{L^\infty (\mathcal{O}_R)}
		\|\nabla \Bu \|_{L^2(\mathcal{O}_R)}.
			\end{aligned}
		\ee
		Similarly, it can be proved that
		\be \label{proof5-50}
		\left| \int_0^1 \int_0^{2\pi} \int_{R-1}^R \partial_z P \Psi_R^z  \, dr d\theta dz \right|
		\leq  C R^{\frac12}(1+\|\Bu\|_{L^\infty (\mathcal{O}_R)})\|u^r\|_{L^\infty (\mathcal{O}_R)} \|\nabla \Bu \|_{L^2(\mathcal{O}_R )}.
		\ee
		
		Collecting the above computations one has
		\begin{equation} \label{proof5-51}
			{Y}(R)  \leq
			C
			\left( R^{\frac12}\|\Bu\|_{L^\infty(\mathcal{O}_R)}
			+
			R^{\frac32}
			(1+  \|\Bu\|_{L^\infty (\mathcal{O}_R)}) \|u^r\|_{L^\infty(\mathcal{O}_R)} \right)
			\left( {Y}^{\prime}(R) \right)^{\frac12} ,
		\end{equation}
		where ${Y}(R)$ is defined in \eqref{localenergy}.
		
	Suppose that  there exist $\beta\in [0, \frac{\alpha}{2}]$ and $C>0$  such that the conditions in \eqref{condthm1.2} hold.  Then for any $\varepsilon>0$, there exists a $R_0 (\varepsilon)>2$ such that
		\begin{equation*}
			\|\Bu\|_{L^\infty (\mathcal{O}_R)} \leq \varepsilon R^{\beta}\quad \text{and}\quad 	\|u^r \|_{L^\infty (\mathcal{O}_R)}\leq C R^{-(1+\beta-\frac{\alpha}{2})} \quad\text{for all }R\geq R_0 (\varepsilon).
		\end{equation*}
		Therefore, inequality \eqref{proof5-51} implies
		\begin{equation*}
			Y(R) \leq C \varepsilon R^{\frac{1+\alpha}{2}} (Y'(R))^{\frac12} \quad\text{for all }R\geq R_0 (\varepsilon).
		\end{equation*}
		
		We assume $\Bu$ is not identically equal to zero.  Then $Y(R)>0$ for large $R$ and it follows that
		\begin{equation*}
			\frac{1}{C \varepsilon} R^{-(1+\alpha)} \leq \frac{Y'}{Y^2}.
		\end{equation*}
		Integrating it over an interval $(R_1,R_2)$ for large $R_1$, we obtain
		\begin{equation*}
			\frac{1}{-\alpha\varepsilon C} (R_2^{-\alpha} -R_1^{-\alpha}) \leq - \frac{1}{Y(R_2)} + \frac{1}{Y(R_1)}.
		\end{equation*}
		According to Part (a) of Theorem \ref{mainresult4}, as $R_2\to\infty$, it holds that $Y(R_2)\to\infty$. Hence
		\begin{equation*}
			\frac{1}{\alpha \varepsilon C} R_1 ^{-\alpha} \leq \frac{1}{Y(R_1)}.
		\end{equation*}
		Therefore,
		\begin{equation*}
			Y(R_1) \leq \alpha \varepsilon C R_1^{\alpha}.
		\end{equation*}
		As $\varepsilon>0$ is chosen arbitrarily, by Part (a) of Theorem \ref{mainresult4}, it follows that $\Bu\equiv 0$, which is a contradiction. Therefore, $\Bu$ is a constant and thus $\Bu\equiv 0$. This completes the proof of Case (b) of Theorem \ref{mainresult4}.
	\end{proof}



Now we are in position to prove Theorem \ref{mainresult5}, where Lemma \ref{bounded-gradient} helps to guarantee the uniform boundedness of gradient of velocity field.
	
	\begin{proof}[{Proof for Theorem \ref{mainresult5}}] Assume that $\Bu$ is a smooth solution to the Navier-Stokes system \eqref{SNS}.
		Taking the $x_i$-derivative($i=1, 2$) of the first equation in \eqref{SNS} one obtains
		\be \label{main5-1}
		-\Delta \partial_{x_i} \Bu + (\partial_{x_i} \Bu \cdot \nabla) \Bu + (\Bu \cdot \nabla ) \partial_{x_i} \Bu + \nabla \partial_{x_i} P = 0.
		\ee
		Multiplying the equation in \eqref{main5-1} by $\varphi_R(r) \partial_{x_i} \Bu$ and integrating over $\Omega$ one derives
		\be \label{main5-2} \ba
		&\int_\Omega |\nabla \partial_{x_i} \Bu|^2 \varphi_R + \int_\Omega (\nabla \varphi_R \cdot \nabla \partial_{x_i} \Bu) \cdot \partial_{x_i} \Bu +\int_\Omega (\partial_{x_i} \Bu \cdot \nabla)\Bu \cdot \partial_{x_i} \Bu \varphi_R \\
		=\, & \frac12 \int_\Omega (\Bu \cdot \nabla \varphi_R) |\partial_{x_i} \Bu |^2 + \int_\Omega \partial_{x_i} P \partial_{x_i} \Bu \cdot \nabla \varphi_R.
		\ea  \ee
		It follows from Poincar\'e inequality that
		\be \label{main5-3}
		\left|  \int_\Omega (\nabla \varphi_R \cdot \nabla \partial_{x_i} \Bu) \cdot \partial_{x_i} \Bu   \right| \leq  \| \nabla \partial_{x_i} \Bu \|_{L^2(\mathcal{O}_R)} \| \partial_{x_i} \Bu\|_{L^2(\mathcal{O}_R)} \leq C R^{\frac12} \| \nabla \partial_{x_i} \Bu \|_{L^2(\mathcal{O}_R)}.
		\ee
		Integration by parts yields
		\be \label{main5-4}
		\int_\Omega (\partial_{x_i} \Bu \cdot \nabla)\Bu \cdot \partial_{x_i} \Bu \varphi_R = -\int_\Omega (\partial_{x_i} \Bu \cdot \nabla) \partial_{x_i} \Bu \cdot \Bu \varphi_R -
		\int_\Omega (\partial_{x_i} \Bu \cdot \nabla \varphi_R) (\Bu \cdot \partial_{x_i} \Bu ).
		\ee
		In fact, one has
		\be \label{main5-5} \ba
		\left| \int_\Omega (\partial_{x_i} \Bu \cdot \nabla) \partial_{x_i} \Bu \cdot \Bu \varphi_R \right|   &\leq  \| \nabla \partial_{x_i} \Bu \sqrt{ \varphi_R} \|_{L^2(\Omega)} \|\partial_{x_i} \Bu \sqrt{\varphi_R } \|_{L^2(\Omega)} \|\Bu \|_{L^\infty(\Omega)} \\
		&\leq \|\nabla \partial_{x_i} \Bu \sqrt{ \varphi_R} \|_{L^2(\Omega)} \frac{1}{\pi} \|\partial_z \partial_{x_i} \Bu \sqrt{\varphi_R } \|_{L^2(\Omega)} \|\Bu \|_{L^\infty(\Omega)} \\
		& \leq \|\nabla \partial_{x_i} \Bu \sqrt{ \varphi_R} \|_{L^2(\Omega)}^2 \|\Bu \|_{L^\infty(\Omega)},
		\ea \ee
		where the Poincar'{e} inequality 
		\[
		\|\partial_{x_i} \Bu \sqrt{\varphi_R } \|_{L^2(\Omega)}\leq \frac{1}{\pi} \|\partial_z (\partial_{x_i} \Bu \sqrt{\varphi_R }) \|_{L^2(\Omega)} =  \frac{1}{\pi} \|\partial_z \partial_{x_i} \Bu \sqrt{\varphi_R } \|_{L^2(\Omega)} 
		\]
	has been used to get the second inequality.
		Furthermore, it holds that
		\be \label{main5-6} \ba
		\left| \int_\Omega (\partial_{x_i} \Bu \cdot \nabla \varphi_R) (\Bu \cdot \partial_{x_i} \Bu )\right|
		\leq \| \partial_{x_i} \Bu \|_{L^2(\mathcal{O}_R) }^2 \| \Bu \|_{L^\infty(\mathcal{O}_R)} \leq  C R^{\frac12} \|\nabla \partial_{x_i} \Bu\|_{L^2(\mathcal{O}_R)},
		\ea \ee
		where Lemma \ref{bounded-gradient} has been used to get the last inequality in \eqref{main5-6}. Similarly, one has
		\be \label{main5-7}
		\left| \int_\Omega \Bu \cdot \nabla \varphi_R |\partial_{x_i} \Bu |^2 \right| \leq \| \Bu\|_{L^\infty(\mathcal{O}_R)} \| \partial_{x_i} \Bu\|_{L^2(\mathcal{O}_R )}^2 \leq  C R^{\frac12} \|\nabla \partial_{x_i} \Bu\|_{L^2(\mathcal{O}_R)}.
		\ee
		Using the moment equation (the first equation in \eqref{SNS}) one obtains
		\be \label{main5-8} \ba
		&\int_\Omega \partial_{x_i} P \partial_{x_i} \Bu \cdot \nabla \varphi_R
		=\int_\Omega \Delta u^i \partial_{x_i} \Bu \cdot \nabla \varphi_R - \int_\Omega (\Bu \cdot \nabla u^i) \partial_{x_i} \Bu \cdot \nabla \varphi_R \\
		=\, & \int_\Omega (\partial_{x_1}^2 + \partial_{x_2}^2) u^i \partial_{x_i} \Bu \cdot \nabla \varphi_R - \int_\Omega \partial_z u^i \partial_z \partial_{x_i} \Bu \cdot \nabla  \varphi_R -  \int_\Omega (\Bu \cdot \nabla u^i) \partial_{x_i} \Bu \cdot \nabla \varphi_R.
		\ea \ee
		Consequently, one has
		\be \label{main5-9}
		\left| \int_\Omega \partial_{x_i} P \partial_{x_i} \Bu \cdot \nabla \varphi_R \right|
		\leq C R^{\frac12} \left( \| \nabla \partial_{x_1} \Bu \|_{L^2(\mathcal{O}_R)} + \|\nabla \partial_{x_2} \Bu\|_{L^2(\mathcal{O}_R)}\right) .
		\ee
		Since $\|\Bu\|_{L^\infty(\Omega)} <\pi$, it holds that
		\be \label{main5-10}
		\int_\Omega |\nabla \partial_{x_i} \Bu|^2 \varphi_R \leq C R^{\frac12} \left( \| \nabla \partial_{x_1} \Bu \|_{L^2(\mathcal{O}_R)} + \|\nabla \partial_{x_2} \Bu\|_{L^2(\mathcal{O}_R)}\right).
		\ee
		
		Define
		\be \nonumber
		Z(R) = \int_\Omega \left( |\nabla \partial_{x_1} \Bu|^2 + |\nabla \partial_{x_2} \Bu|^2 \right) \varphi_R\left(\sqrt{x_1^2+x_2^2}\right)\,dx.
		\ee
		Hence the estimate \eqref{main5-10} implies
		\be \label{main5-12}
		Z(R) \leq CR^{\frac12} Z^{\prime}(R)^{\frac12}.
		\ee
		Assume $\nabla \partial_{x_1}\Bu$ and $\nabla \partial_{x_2}\Bu$ are not identically equal to zero. Then $Z(R)>0$ for $R\geq R_0$ with $R_0>0$, and one has
		\begin{equation*}
			\frac{1}{CR}\leq \left( -\frac{1}{Z(R)}\right)'.
		\end{equation*}
		Integrating it over $(R_0,R)$ for large $R_0$ one arrives at
		\begin{equation*}
			\frac{1}{C}\ln{\frac{R}{R_0}} \leq -\frac{1}{Z(R)}+\frac{1}{Z(R_0)}\leq \frac{1}{Z(R_0)}.
		\end{equation*}
	This leads to a contradiction when $R$ is sufficiently large. Therefore,  $\nabla \partial_{x_1} \Bu = \nabla \partial_{x_2} \Bu \equiv 0$. Note that $\partial_{x_1}\Bu=\partial_{x_2}\Bu=0$ at the boundary $\mathbb{R}^2\times \{0, 1\}$.  Thus $\partial_{x_1} \Bu = \partial_{x_2} \Bu \equiv 0$. It follows from the divergence free property of $\Bu$ that $\partial_{x_3} u^3=0$. This, together with the no-slip boundary conditions, yields that
		\be \nonumber
		u^{1} = u^1(x_3), \ \  u^2 = u^2(x_3), \ \  \text{and}\,\, u^3 \equiv 0.
		\ee
		Hence Navier-Stokes system can be written as
		\be \nonumber
		\partial_z^2 u^1 + \partial_{x_1} P = \partial_z^2 u^2 + \partial_{x_2} P =\partial_{x_3} P= 0.
		\ee
		Taking the homogeneous boundary condition of $\Bu$ into consideration one derives
		\be \nonumber
		u^1 = c_1 x_3(1-x_3) \ \ \text{and} \ \ u^2 = c_2x_3(1-x_3) \ \ \ \ \text{for some}\,\, c_1, c_2\in \mathbb{R}.
		\ee
		This completes the proof of Theorem \ref{mainresult5}.
	\end{proof}

	\section{Liouville type theorem for flows in a periodic slab}\label{Sec5}
	
	This section is devoted to the proof of Theorem \ref{corollary2}. Before the detailed presentation for the proof, let us introduce several lemmas which give some important properties for solutions of Navier-Stokes system in $\mathbb{R}^2\times \mathbb{T}$. We first prove that the pressure is also periodic once the  solution $\Bu$ is uniformly bounded and periodic in one direction.
	\begin{lemma}\label{pressureperiodic-new} Let $\Bu$ be a bounded smooth solution to the Navier-Stokes system \eqref{SNS} in $\mathbb{R}^2 \times \mathbb{T}$. The pressure $P$ is also a periodic function with respect to $z$.
	\end{lemma}
	\begin{proof} Let $Q(r, \theta, z) = P(r, \theta,  z+ 1) - P(r, \theta,  z).$ Since $\Bu$ is periodic with respect to $z$, one has
		\be \nonumber
		\nabla Q = \nabla P(r, \theta, z+1) - \nabla P(r, \theta, z) = 0 .
		\ee
		This implies that $Q= Q_0$ for some constant $Q_0 \in \mathbb{R}$.
		
		Integrating the third equation in \eqref{NScylindrical} with respect to $(\theta, z)$ on $[0, 2\pi]\times [0,1]$ and integration by parts one obtains
		\be \label{pp-2}\ba
		&2\pi Q_0  =\int_0^1 \int_0^{2\pi}\partial_z P(r, \theta, z)d\theta dz & \\
		= \,&\int_0^1 \int_0^{2\pi} \left( \partial_r^2 u^z + \frac{1}{r} \partial_r u^z \right)  \, d\theta  dz -\int_0^1 \int_0^{2\pi} u^r \partial_r u^z-\left(\frac{1}{r}\partial_\theta u^\theta+\partial_z u^z\right)u^z \, d\theta dz\\
		=\, & \left( \partial_r + \frac{1}{r} \right) \int_0^1 \int_0^{2\pi} \partial_r u^z \, d\theta dz  - \left(  \partial_r + \frac{1}{r} \right) \int_0^1 \int_0^{2\pi} u^r u^z \, d\theta dz ,
		\ea \ee
		where the divergence free property of $\Bu$ (the last equation in \eqref{NScylindrical}) has been used to get the last equality.
		Hence one has
		\be \label{pp-3}
		r \int_0^1 \int_0^{2\pi} \partial_r u^z \, d\theta dz -  r \int_0^1 \int_0^{2\pi}  u^r u^z \, d\theta dz = \pi Q_0 r^2 + Q_1,
		\ee
		for some constant $Q_1 \in \mathbb{R}$.
		
		Meanwhile, since $\Bu$ is bounded in $\mathbb{R}^2\times \mathbb{T}$, it follows from Lemma \ref{bounded-gradient} that $\nabla \Bu$ is also  bounded in $\mathbb{R}^2\times \mathbb{T}$. Hence there exists a constant $C>0$ such that
		\be \nonumber
		\left| r \int_0^1 \int_0^{2\pi} \partial_r u^z \, d\theta dz -  r \int_0^1 \int_0^{2\pi} u^r u^z \, d\theta dz \right| \leq Cr.
		\ee
		Hence $Q_0=0$. This implies that $P$ is periodic with respect to $z$.
	\end{proof}

	\begin{remark}
		In fact, the pressure of the Navier-Stokes system \eqref{SNS} in a periodic slab may not be  periodic  when the velocity is periodic but not bounded. For example, $\bBu = r^2 \Be_z$ is a solution of Navier-Stokes equations in $\mathbb{R}^2 \times \mathbb{T}$, however, the associated pressure $P = 4z$ is not periodic with respect to $z$.
	\end{remark}


	The following lemma shows Liouville-type theorem for solutions of the Navier-Stokes system  in a periodic slab when the associated Dirichlet integral is finite.
	\begin{lemma}\label{periodic-Dirichlet}
		Let $\Bu$ be a bounded  smooth solution to the Navier-Stokes system \eqref{SNS} in $\mathbb{R}^2\times \mathbb{T}$. Then $\Bu = (0, 0, c)$, provided that $u^\theta $ is independent of $\theta$ and $\Bu$ has a finite Dirichlet integral in the slab, i.e.,
		\be \label{periodic-1}
		\int_{\mathbb{R}^2\times (0,1)} |\nabla \Bu|^2 < +\infty.
		\ee
	\end{lemma}
	
	\begin{remark} In fact, the results in Lemma  \ref{periodic-Dirichlet} have been obtained in \cite{Pan}. Here we give a
		different and simpler proof, which contains some of the key ingredients for the analysis on
		general solutions whose Dirichlet integrals may not be finite.
	\end{remark}
	
	\begin{proof}[{Proof of Lemma \ref{periodic-Dirichlet}.]  The proof contains two steps.
		
		{\it 	Step 1. Set up.} Since $\Bu$ is a bounded smooth solution to \eqref{SNS} in $\mathbb{R}^2\times \mathbb{T}$, it follows from
			Lemma \ref{pressureperiodic-new} that the equality \eqref{estimate1} still holds.

			Due to the divergence free property of $\Bu$ and the fact that $u^\theta$ is independent of $\theta$, one has
			\be \nonumber
			\partial_r \int_0^1 r u^r \, dz = - \int_0^1 \partial_z (r u^z ) \, dz = 0, \ \ \ \ \text{for all} \,\, 0\leq r < \infty, \ \ 0\leq \theta \leq 2\pi.
			\ee
			This implies
			\be \label{lemma3-3}
			\int_0^1 r u^r \, dz =  0 \quad \text{and}\quad  \int_0^1 u^r \, dz =0.
			\ee
		It follows from   \eqref{lemma3-3} and Lemma \ref{Bogovskii} that for every fixed $\theta \in [0, 2\pi]$, there exists a vector valued function $\BP_{R, \theta}(r, z) \in H_0^1(D_R; \mathbb{R}^2)$ satisfying
			\be \label{lemma3-5}
			\partial_r \Psi_{R, \theta}^r + \partial_z \Psi_{R, \theta}^z =  r u^r
			\ee
			together with the estimate
			\begin{equation}\label{lemma3-5.1}
				\| \partial_r \BP_{R, \theta} \|_{L^2(D_R)} + \| \partial_z \BP_{R, \theta} \|_{L^2(D_R ) } \leq C \|r u^r \|_{L^2(D_R)},
			\end{equation}
			where $C$ is independent of $\theta$. It follows from 	   \eqref{lemma3-3} that
				\be \label{lemma3-3.5}
			\int_0^1 r \partial_\theta u^r \, dz =  0 .
			\ee
			Note that the Bogovskii map is a linear map (\hspace{1sp}\cite{Galdi}). Hence there is a universal constant $C>0$ such that
			\begin{equation}
				\label{lemma3-5-1}
				\|\partial_\theta \partial_r \BP_{R, \theta} \|_{L^2(D_R)} + \| \partial_\theta \partial_z \BP_{R, \theta} \|_{L^2(D_R ) } \leq C \| r \partial_\theta u^r \|_{L^2(D_R)} .
			\end{equation}
				By Poincar\'e inequality, one has
			\be \label{lemma3-3-1}
			\|u^r \|_{L^2(\mathcal{O}_R )} \leq C \|\partial_z u^r \|_{L^2(\mathcal{O}_R )}.
			\ee
			This, together with \eqref{lemma3-5.1} and \eqref{lemma3-5-1}, gives
			\be \label{lemma3-6}
			\| \partial_r \BP_{R, \theta} \|_{L^2(\mathcal{D}_R)} + \| \partial_z \BP_{R, \theta} \|_{L^2(\mathcal{D}_R ) } \leq C \|r u^r \|_{L^2(\mathcal{D}_R)}
			\leq
			C R^{\frac12}\|\nabla \Bu\|_{L^2 (\mathcal{O}_R)}
			\ee
			and
			\be \label{lemma3-7}
			\|\partial_\theta \partial_r \BP_{R, \theta} \|_{L^2(\mathcal{D}_R)} + \| \partial_\theta \partial_z \BP_{R, \theta} \|_{L^2(\mathcal{D}_R ) } \leq C \| r \partial_\theta u^r \|_{L^2(\mathcal{D}_R)}
			\leq
			C R^{\frac32} \|  \nabla \Bu \|_{L^2(\mathcal{O}_R)}.
			\ee

		Furthermore, it follows from Lemma \ref{pressureperiodic-new} and \eqref{lemma3-5} that one has
			\be \label{lemma3-11}
			\ba  \int_{\Omega}  P\Bu \cdot \nabla\varphi_{R} & = - \int_0^1 \int_0^{2\pi} \int_{R-1}^R P \cdot r u^r \, dr d\theta dz \\
			& = - \int_0^1 \int_0^{2\pi} \int_{R-1}^R P \left(  \partial_r \Psi_{R, \theta}^r + \partial_z \Psi_{R, \theta}^z     \right)  \, dr d\theta dz \\
			& = \int_0^1 \int_0^{2\pi} \int_{R-1}^R ( \partial_r P \Psi_{R, \theta}^r + \partial_z P \Psi_{R, \theta}^z ) \, dr d\theta dz .
			\ea
			\ee
			Clearly, the right hand side of \eqref{lemma3-11} can be represented by equations \eqref{proof4-20} and \eqref{proof4-20.5} with
			$(\Psi^r_R, \Psi^z_R)$ replaced by	$(\Psi^r_{R,\theta}, \Psi^z_{R,\theta})$.
			
			{\it Step 2. Saint-Vernant type estimate.} We estimate the first two terms on the right hand
			side of \eqref{estimate1} and the right hand side of \eqref{lemma3-11}. 	Using H\"{o}lder inequality one obtains
			\be \label{lemma3-2} \left|   \int_{\Omega} \nabla \varphi_R \cdot \nabla \Bu \cdot \Bu  \right|
			\leq C \| \nabla \Bu \|_{L^2(\mathcal{O}_R)} \cdot R^{\frac12} \| \Bu \|_{L^\infty(\mathcal{O}_R)}
			\leq C R^{\frac12} \| \nabla \Bu \|_{L^2(\mathcal{O}_R)} .
			\ee
It follows from \eqref{lemma3-3-1} that
				\be \label{lemma3-4}
			\left| \int_{\Omega}   \frac12 |\Bu|^2\Bu \cdot \nabla \varphi_R \right| \leq C \|u^r \|_{L^2(\mathcal{O}_R)} \cdot R^{\frac12} \| \Bu \|_{L^\infty(\mathcal{O}_R)}^2
			\leq C R^{\frac12} \| \nabla \Bu \|_{L^2(\mathcal{O}_R)}.
			\ee

			By virtue of Poincar\'e inequality \eqref{lemma3-3-1} and the estimates  \eqref{lemma3-6}-\eqref{lemma3-7}, one has
			\be \label{lemma3-15}
			\ba
			& \left| \int_0^1 \int_0^{2\pi} \int_{R-1}^R (\partial_r u^r  \partial_r \Psi_{R, \theta}^r + \partial_z u^r \partial_z \Psi_{R, \theta}^r )\, dr d\theta dz \right| \\
			\leq\, & C R^{-\frac12} \|\nabla \Bu\|_{L^2(\mathcal{O}_R )} \cdot R^{\frac12} \|u^r \|_{L^2(\mathcal{O}_R )} \leq CR^{\frac12} \|\nabla \Bu\|_{L^2(\mathcal{O}_R)}
			\ea
			\ee
			and
			\be \label{lemma3-16}
			\ba
			\left| \int_0^1 \int_0^{2\pi} \int_{R-1}^R \frac{1}{r^2} \partial_\theta u^r \partial_\theta \Psi_{R, \theta}^r \, dr d\theta dz \right|
			\leq\, & C  R^{-\frac32} \|\nabla \Bu\|_{L^2(\mathcal{O}_R )} \cdot R^{\frac12} \|\partial_\theta u^r \|_{L^2(\mathcal{O}_R)} \\
			\leq\, & C \|\nabla \Bu \|_{L^2(\mathcal{O}_R )}^2
			\leq C \|\nabla \Bu \|_{L^2(\mathcal{O}_R)},
			\ea
			\ee
			where the last inequality is due to the assumption \eqref{periodic-1}. Furthermore, one has
			\be \label{lemma3-17}
			\ba
			& \left| \int_0^1 \int_0^{2\pi} \int_{R-1}^R \left( \frac1r \partial_r - \frac{1}{r^2} \right) u^r \Psi_{R, \theta}^r \, drd\theta dz \right| \\
			\leq\, & C R^{-1} R^{-\frac12} \|\nabla \Bu \|_{L^2(\mathcal{O}_R )} \cdot R^{\frac12} \| u^r\|_{L^2(\mathcal{O}_R)} \leq C R^{-\frac12} \|\nabla \Bu \|_{L^2(\mathcal{O}_R )}
			\ea
			\ee
			and
			\be \label{lemma3-18}
			\ba
			& \left| \int_0^1 \int_0^{2\pi} \int_{R-1}^R \left[ \left(u^r \partial_r +\frac{u^\theta}{r}\partial_\theta + u^z \partial_z \right) u^r - \frac{(u^\theta)^2 }{r} \right] \Psi_{R, \theta}^r \, dr d\theta dz\right| \\
			\leq\, & C \|\Bu \|_{L^\infty(\mathcal{O}_R )} \left(  R^{-\frac12}\|\nabla \Bu \|_{L^2(\mathcal{O}_R)} +  R^{-\frac32} \|u^\theta\|_{L^2(\mathcal{O}_R )} \right)\cdot R^{\frac12} \|u^r \|_{L^2(\mathcal{O}_R)} \\
			\leq\, & C R^{\frac12} \|\nabla \Bu \|_{L^2(\mathcal{O}_R ) }.
			\ea
			\ee
			Collecting the estimates \eqref{lemma3-15}-\eqref{lemma3-18} one derives
			\be \label{lemma3-19}
			\left|  \int_0^1 \int_0^{2\pi} \int_{R-1}^R \partial_r P \Psi_{R, \theta}^r \, dr d\theta dz \right| \leq C R^{\frac12} \|\nabla \Bu\|_{L^2(\mathcal{O}_R ) }.
			\ee
			Similarly, it holds that
			\be \label{lemma3-20}
			\left|  \int_0^1 \int_0^{2\pi} \int_{R-1}^R \partial_z P \Psi_{R, \theta}^z \, dr d\theta dz \right| \leq C R^{\frac12} \|\nabla \Bu\|_{L^2(\mathcal{O}_R ) }.
			\ee
	Combining \eqref{lemma3-2}--\eqref{lemma3-4} and \eqref{lemma3-19}--\eqref{lemma3-20} one arrives at
			\begin{equation}\label{lem3-4ODE}
				Y(R) \leq CR^{\frac12}[Y'(R)]^{\frac12}
			\end{equation}
			where $Y(R)$ is defined in \eqref{localenergy}.
			The same argument as that for the proof of Theorem \ref{mainresult5} proves that $\nabla \Bu \equiv 0$. Thus $\Bu$ is a constant vector. Since $u^\theta$ is independent of $\theta$, one has $\Bu=(0, 0, c)$ for some constant $c$.
		\end{proof}
		
		Now we are ready for the proof of Theorem \ref{corollary2}.

		\begin{proof}[Proof for Theorem \ref{corollary2}}]
		 Since $\Bu$ is a bounded smooth solution to \eqref{SNS} in $\mathbb{R}^2\times \mathbb{T}$, it follows from
		Lemma \ref{pressureperiodic-new} that the equality \eqref{estimate1} still holds.
		
		We divide the rest of proof into three steps.
		
		{\it Step 1. Proof for Case (a) of Theorem \ref{corollary2}}.
		The proof is almost the same as that for Lemma \ref{periodic-Dirichlet}, except that
		\be \label{lemma3-31}
		\ba
		\left| \int_0^1 \int_0^{2\pi} \int_{R-1}^R \frac{1}{r^2} \partial_\theta u^r \partial_\theta \Psi_{R, \theta}^r \, dr d\theta dz \right|
		\leq\,& C R^{-\frac32} \|\nabla \Bu\|_{L^2(\mathcal{O}_R )} \cdot R^{\frac12} \|\partial_\theta u^r \|_{L^2(\mathcal{O}_R)}\\
		\leq \,& C \|\nabla \Bu \|_{L^2(\mathcal{O}_R )}^2.
		\ea \ee
		The computations in the proof of Lemma \ref{periodic-Dirichlet} imply
		\be \label{lemma3-41}
		{Y}(R) \leq C_1 {Y}^{\prime}(R) + C_2 R^{\frac12} \left[{Y}^{\prime}(R)\right]^{\frac12},
		\ee
		where $Y(R)$ is defined in \eqref{localenergy}.
		Hence one has
		\be \label{lemma3-42}
		\left[{Y}^{\prime}(R)\right]^{\frac12} \geq \frac{-C_2 R^{\frac12} +\sqrt{C_2^2 R + 4C_1 Y(R) } }{2C_1} \geq \frac{Y(R)}{ \sqrt{C_2^2 R + 4C_1 Y(R) }}.
		\ee
		Suppose that $\nabla \Bu$ is not identically equal to zero. For $R$ large enough, $Y(R) > 0$,
		\be \label{lemma3-43}
		\left[ C_2^2 R Y^{-2}(R) + 4C_1 Y^{-1}(R)     \right] Y^{\prime}(R) \geq 1,
		\ee
		Let $M$ be a large number satisfying $M^{-1} C_2^2 \leq \frac14$. According to Lemma \ref{periodic-Dirichlet}, there exists an $R_0 >2$ such that $Y(R_0) \geq M$, otherwise $\nabla \Bu \equiv 0$. For every $R > R_0$, integrating \eqref{lemma3-43} over $[R, 2R]$, one gets
		\be \label{lemma3-51}
		2R \cdot C_2^2 \left[ \frac{1}{Y(R) } - \frac{1}{Y(2R)}    \right] + 4C_1 \ln \frac{Y(2R)}{Y(R)} \geq R.
		\ee
		Since $Y(R) \geq M$, it holds that
		\be \label{lemma3-52}
		\frac{Y(2R)}{Y(R)} \geq \exp \left\{ \frac{R}{8 C_1}   \right\}.
		\ee
		This implies the exponential growth of $\| \nabla \Bu \|_{L^2(\Omega_R)}$ and leads to a contradiction to the uniform boundedness of $\nabla \Bu$, according to Lemma \ref{bounded-gradient}. Hence $\nabla \Bu \equiv 0$ and $\Bu = (0, 0, c)$, since $u^\theta$ is axisymmetric. The proof for Case (a) of  Theorem \ref{corollary2} is completed.

	
{\it Step 2. Proof for Case (b) of Theorem \ref{corollary2}}.	Using the divergence free property of $\Bu$, one has
\be \nonumber
\partial_r \int_0^1\int_0^{2\pi} r u^r \,d\theta dz = - \int_0^1\int_0^{2\pi} \partial_\theta u^\theta +\partial_z (r u^z ) \,d\theta dz = 0, \ \ \ \ \text{for all} \,\, 0\leq r < \infty.
\ee
Since $u^r$ is independent of $\theta$, it holds that
\be \label{lemma3-3_2.1}
\int_0^1 r u^r \, dz = \frac{1}{2\pi}\int_0^1\int_0^{2\pi} r u^r \,d\theta dz=  0 \quad \text{and}\quad  \int_0^1 u^r \, dz =0.
\ee
Hence we have Poincar\'{e} inequality \eqref{lemma3-3-1}. Furthermore, 	by virtue of Lemma \ref{Bogovskii}, there exists a vector valued function $\BP_R(r, z) \in H_0^1(D_R; \mathbb{R}^2)$ satisfying
the equation \eqref{proof-9}
 and the estimate
\eqref{proof-10}. Therefore, one has
	\be \label{lemma3-3.2.2}
	\begin{aligned}
\int_{\Omega} P \Bu \cdot \nabla \varphi_R
= & \int_0^1 \int_{R-1}^R \int_0^{2\pi} P u^r r \, d\theta dr dz
  =  \int_0^1 \int_{R-1}^R\int_0^{2\pi} P \left( \partial_r \Psi_R^r + \partial_z \Psi_R^z  \right) \,d\theta drdz\\
  = & 	 - \int_0^1 \int_{R-1}^R\int_0^{2\pi} (\partial_r P \Psi_R^r + \partial_z P \Psi_R^z ) \, d\theta drdz,
\end{aligned}
\ee
where the right hand side can be represented by \eqref{proof4-20} and \eqref{proof4-20.5}

Now let us start the estimate for the right hand side of \eqref{estimate1}.
 The first two terms on the right hand side of \eqref{estimate1} can be estimated as the  same  as that in \eqref{lemma3-2} and \eqref{lemma3-4}. Note that we do not need to estimate the term appeared on the left hand side of \eqref{lemma3-16} since $\partial_\theta u^r\equiv 0$. All the other terms appeared in  \eqref{proof4-20} and \eqref{proof4-20.5} can be estimated  in the exactly same way as that in Step 2 of the proof for Lemma \ref{periodic-Dirichlet}. Hence we arrive at \eqref{lem3-4ODE} with $Y(R)$ defined in \eqref{localenergy}. As in the proof for Theorem \ref{mainresult5}, we can show that $\nabla \Bu \equiv 0$. Thus $\Bu$ is a constant vector. Since $u^r$ is independent of $\theta$, one has $\Bu=(0, 0, c)$ for some constant $c$. This finishes the proof for Case (b) of Theorem \ref{corollary2}.

			{\it Step 3. Proof for Case (c) of Theorem \ref{corollary2}}. For steady solutions of Navier-Stokes system in $\mathbb{R}^2\times \mathbb{T}$,  following the same proof as that for \eqref{proof4-12} one obtains
		\be \nonumber
		\int_0^1 \int_0^{2\pi}  r u^r \,  d\theta dz = 0,  \,\,\int_0^1 \int_0^{2\pi} u^r \, d\theta dz = 0, \,\,\text{and}\,\,  \int_0^1 \int_0^{2\pi}  \int_{R-1}^R  r u^r \, dr d\theta dz = 0.
		\ee
		By Poincar\'e's inequality, one has
		\be \label{coro4-2-1}
		\|u^r \|_{L^2(\mathcal{O}_R )} \leq C R^{\frac12} \| u^r\|_{L^2(\mathcal{D}_R)} \leq C R^{\frac12} \| (\partial_r, \partial_\theta, \partial_z) u^r \|_{L^2(\mathcal{D}_R)} \leq C R \|\nabla \Bu\|_{L^2(\mathcal{O}_R)}.
		\ee
		And there exists a vector valued function $\BP_R(r, \theta, z)\in H_0^1 (\mathcal{D}_R; \mathbb{R}^3)$ satisfying  \eqref{proof4-13}-\eqref{proof4-15}. Thus it holds that
		\be \label{coro4-3} \ba
		\int_{\Omega} P \Bu \cdot \nabla \varphi_R & = \int_0^1 \int_0^{2\pi} \int_{R-1}^R P u^r r \, dr d\theta dz \\
		& = -   \int_0^1 \int_0^{2\pi} \int_{R-1}^R (\partial_r P \Psi_R^r + \partial_\theta P \Psi_R^\theta +  \partial_z P \Psi_R^z ) \, dr d\theta dz.
		\ea  \ee
	Clearly, we have the same equations \eqref{proof4-20}--\eqref{proof4-20.5} to characterize the terms on the right hand side of \eqref{coro4-3}.
	
	Now we are in position to estimate the first two terms on the right hand side of \eqref{estimate1} and the terms on \eqref{coro4-3} carefully.
		First,  one has
		\be \label{coro4-1}
		\left| \int_{\Omega} (\nabla \varphi_R \cdot \nabla \Bu) \cdot \Bu  \right|
		\leq C \|\nabla \Bu \|_{L^2(\mathcal{O}_R)} \|\Bu\|_{L^2(\mathcal{O}_R ) }
		\leq C  R^{\frac12} \| \nabla \Bu \|_{L^2(\mathcal{O}_R)}
		\ee
		and
		\be \label{coro4-2} \ba   \left| \int_{\Omega}  \frac12 |\Bu|^2\Bu \cdot \nabla \varphi_R  \right|
		& \leq C  \| \Bu \|_{L^\infty (\mathcal{O}_R ) }^2 \| u^r \|_{L^1  (\mathcal{O}_R ) }  \leq C R \|u^r \|_{L^\infty(\mathcal{O}_R )}.
		\ea
		\ee
		According to \eqref{proof4-15}, one has
		\be \label{coro4-21}
		\ba
		& \left| \int_0^1 \int_0^{2\pi} \int_{R-1}^R \left(\partial_r u^r \partial_r \Psi_R^r + \partial_z u^r \partial_z \Psi_R^r +  \frac{1}{r^2} \partial_\theta u^r \partial_\theta \Psi_R^r \right) \, drd\theta dz \right| \\
		\leq\,\,& C R^{-\frac12} \| \nabla \Bu \|_{L^2(\mathcal{O}_R)} \cdot R^{\frac12} \| u^r \|_{L^2(\mathcal{O}_R)} \leq C  R^{\frac12} \|\nabla \Bu \|_{L^2(\mathcal{O}_R)} \|u^r \|_{L^\infty(\mathcal{O}_R ) }.
		\ea
		\ee
		and
		\be \label{coro4-22} \ba
		& \left|  \int_0^1 \int_0^{2\pi} \int_{R-1}^R \left[ \left( \frac{1}{r} \partial_r   - \frac{1}{r^2}       \right)u^r + \frac{2}{r^2}\partial_\theta u^\theta \right]  \Psi_R^r  \, dr d\theta dz      \right| \\
		\leq\,\,& C R^{-\frac32}  \|\nabla \Bu\|_{L^2(\mathcal{O}_R)}  \cdot   R^{\frac12} \|u^r \|_{L^2(\mathcal{O}_R)}  \leq C R^{-\frac12}    \| \nabla \Bu\|_{L^2(\mathcal{O}_R ) } \| u^r \|_{L^\infty(\mathcal{O}_R ) }.
		\ea \ee
		Furthermore, it holds that
		\be \label{coro4-23}
		\ba
		&  \left| \int_0^1 \int_0^{2\pi} \int_{R-1}^R \left(u^r \partial_r + \frac{1}{r} u^\theta \partial_\theta  + u^z \partial_z \right) u^r \Psi_R^r \, dr d\theta  dz \right| \\
		\leq\, & C  \| (u^r, u^\theta, u^z) \|_{L^\infty (\mathcal{D}_R) } \| (\partial_r, r^{-1} \partial_\theta , \partial_z )u^r \|_{L^2(\mathcal{D}_R)}
		\| \Psi_R^r \|_{L^2(\mathcal{D}_R)} \\
		\leq \,& C R^{\frac12} \|\nabla \Bu \|_{L^2(\mathcal{O}_R )} \| u^r \|_{L^\infty(\mathcal{O}_R)}
		\ea
		\ee
		and
		\be \label{coro4-26}  \left| \int_0^1 \int_0^{2\pi} \int_{R-1}^R \frac{(u^\theta)^2}{r}  \Psi_R^r \, dr d\theta  dz \right|  \leq CR^{-1} \|u^\theta \|_{L^\infty (\mathcal{D}_R)}^2 \| \Psi_R^r \|_{L^1 (\mathcal{D}_R )}  \leq C R^{\frac12} \| \nabla \Bu \|_{L^2(\mathcal{O}_R )},
		\ee
		where the last inequality is a consequence of \eqref{coro4-2-1}.		
		Combining the estimates \eqref{coro4-21}-\eqref{coro4-26} one obtains
		\be \label{coro4-27}  \left| \int_0^1 \int_0^{2\pi} \int_{R-1}^R \partial_r P \Psi_R^r \, dr d\theta dz \right|
		\leq C R^{\frac12} \| \nabla \Bu \|_{L^2(\mathcal{O}_R )} .
		\ee

		It follows from \eqref{proof4-15} that one has
		\be \label{coro4-31}
		\ba
		&  \int_0^1 \int_0^{2\pi} \int_{R-1}^R \left[ r(\partial_r u^\theta \partial_r \Psi_R^\theta + \partial_z u^\theta \partial_z \Psi_R^\theta) +  r^{-1} \partial_\theta u^\theta \partial_\theta \Psi_R^\theta \right] \, drd\theta dz \\
		\leq \,& C R \| (\partial_r u^\theta, \partial_z u^\theta, r^{-2} \partial_\theta  u^\theta ) \|_{L^2(\mathcal{D}_R)} \| \tilde{\nabla} \Psi_R^\theta \|_{L^2(\mathcal{D}_R)} \\
		\leq\, & C R^{\frac12} \| \nabla \Bu \|_{L^2(\mathcal{O}_R)}\cdot R^{\frac12} \| u^r \|_{L^2(\mathcal{O}_R) } \\
		\leq\, & C R^{\frac32} \|\nabla \Bu \|_{L^2(\mathcal{O}_R)} \| u^r \|_{L^\infty(\mathcal{O}_R ) }
		\ea
		\ee
		and
		\be \label{coro4-32}
		\ba
		& \left| \int_0^1 \int_0^{2\pi} \int_{R-1}^R \left[ - \frac{1}{r}     u^\theta  + \frac{2}{r} \partial_\theta u^r \right]\Psi_R^\theta  \, dr d\theta dz \right|\\
		\leq\, & C  \left(  R^{-1} \| u^\theta \|_{L^2(\mathcal{D}_R)} + \|r^{-1} \partial_\theta u^r \|_{L^2(\mathcal{D}_R)} \right) \| \Psi_R^\theta \|_{L^2(\mathcal{D}_R)}\\
		\leq\, & C \left( R^{-1 } \|u^\theta \|_{L^\infty (\mathcal{O}_R)} + R^{-\frac12} \| \nabla \Bu \|_{L^2(\mathcal{O}_R ) } \right)  \cdot  R^{\frac12} \| u^r \|_{L^2(\mathcal{O}_R) }\\
		\leq\, & C R^{\frac12} \|\nabla \Bu \|_{L^2(\mathcal{O}_R )}.
		\ea
		\ee
		Furthermore, applying Poincar\'{e} inequality one derives
		\be \label{coro4-33} \ba
		\left| \int_0^1 \int_0^{2\pi} \int_{R-1}^R  u^r u^{\theta} \Psi_R^{\theta}   \, dr d\theta dz \right|
		& \leq C \|u^r \|_{L^\infty(\mathcal{D}_R)} \|u^\theta \|_{L^\infty(\mathcal{D}_R)} \| \Psi_R^\theta \|_{L^1(\mathcal{D}_R)} \\
		& \leq C R^{\frac32} \| \nabla \Bu \|_{L^2(\mathcal{O}_R )} \| u^r \|_{L^\infty(\mathcal{O}_R)}
		\ea  \ee
		and
		\be  \label{coro4-35} \ba
		& \left| \int_0^1 \int_0^{2\pi} \int_{R-1}^R \left(r u^r \partial_r +  u^\theta \partial_\theta  + r u^z \partial_z \right) u^\theta \Psi_R^\theta  \, dr d\theta  dz \right| \\
		\leq\, & C R \| (u^r, u^\theta, u^z) \|_{L^\infty(\mathcal{D}_R)} \| (\partial_r, r^{-1} \partial_\theta, \partial_z )u^\theta \|_{L^2(\mathcal{D}_R)} \| \Psi_R^\theta \|_{L^2(\mathcal{D}_R )} \\
		\leq\, & C R^{\frac32}  \| \nabla \Bu \|_{L^2(\mathcal{O}_R)} \|u^r \|_{L^\infty (\mathcal{O}_R )}.
		\ea
		\ee
		Combining the estimates \eqref{coro4-31}-\eqref{coro4-35} one gets
		\be \label{coro4-38}
		\left| \int_0^1 \int_0^{2\pi} \int_{R-1}^R \partial_\theta P \Psi_R^\theta  \, dr d\theta dz \right|
		\leq C R^{\frac12} \| \nabla \Bu \|_{L^2(\mathcal{O}_R)} (1  + R \|u^r \|_{L^\infty(\mathcal{O}_R )} ) .  \ee
		
		Similarly, it can be proved that
		\be \label{coro4-50}  \left| \int_0^1 \int_0^{2\pi} \int_{R-1}^R \partial_z P \Psi_R^z  \, dr d\theta dz \right|
		\leq CR^{\frac12} \| \nabla \Bu \|_{L^2(\mathcal{O}_R )}.
		\ee
		
		Since $ru^r$ is bounded, one obtains that
		\be \label{coro-51}
		Y(R) \leq C_1 R \|u^r \|_{L^\infty(\mathcal{O}_R)} + C_2 R^{\frac12} \left[Y^{\prime}(R)\right]^{\frac12},
		\ee
		where $Y(R)$ is defined in \eqref{localenergy}.
		Suppose that $\nabla \Bu$ is not identically equal to zero, there exists an $R_0$ large enough, such that $Y(R_0)>0$. Since $ru^r$ converges to zero, there exists some $R_1> R_0$ such that $ Y(R_0) \geq 2C_1 R  \|u^r\|_{L^\infty(\mathcal{O}_R)}$ for every $R \geq R_1$ . Hence it holds that
		\be \nonumber
		Y(R) \leq 2C_2 R^{\frac12} \left[Y^{\prime}(R)\right]^{\frac12}, \ \ \ R\geq R_1,
		\ee
		which leads to contradiction. Hence $\Bu = (0, 0, c)$.

			{\it Step 4. Proof for Case (d) of Theorem \ref{corollary2}}. Assume that $\Bu$ is a bounded, smooth solution to the Navier-Stokes system \eqref{SNS} in $\mathbb{R}^2\times \mathbb{T}$. Taking the $x_3$-derivative of the momentum equation,
\be \label{6-1}
-\Delta \partial_{x_3} \Bu + (\partial_{x_3} \Bu \cdot \nabla)\Bu + (\Bu \cdot \nabla) \partial_{x_3} \Bu + \nabla \partial_{x_3} P = 0.
\ee
Multiplying the equation \eqref{6-1} by $\varphi_R \partial_{x_3} \Bu$ and integrating over $\Omega$, one has
\be \ba
& \int_\Omega |\nabla \partial_{x_3} \Bu|^2 \varphi_R + \int_\Omega \nabla \varphi_R \cdot \nabla \partial_{x_3} \Bu \cdot \partial_{x_3} \Bu + \int_\Omega (\partial_{x_3} \Bu \cdot \nabla) \Bu \cdot \partial_{x_3} \Bu \varphi_R \\
=\,\, &\frac12 \int_{\Omega} (\Bu \cdot \nabla \varphi_R) |\partial_{x_3} \Bu|^2 + \int_{\Omega} \partial_{x_3} P   \partial_{x_3} \Bu  \cdot \nabla  \varphi_R.
\ea \ee
Since $\int_0^1 \partial_{x_3} \Bu \, dx_3 = 0$, by virtue of Poincar\'{e} inequality and Lemma \ref{bounded-gradient}, one has
\be \label{6-5}
\left| \int_\Omega \nabla \varphi_R \cdot \nabla \partial_{x_3} \Bu \cdot \partial_{x_3} \Bu \right| \leq \|\nabla \partial_{x_3} \Bu \|_{L^2(\mathcal{O}_R)} \| \partial_{x_3} \Bu\|_{L^2(\mathcal{O}_R)}\leq C R^{\frac12} \|\nabla \partial_{x_3} \Bu \|_{L^2(\mathcal{O}_R)}
\ee
and
\be \label{6-6}
\int_\Omega (\partial_{x_3} \Bu \cdot \nabla) \Bu \cdot \partial_{x_3} \Bu \varphi_R = - \int_{\Omega} (\partial_{x_3} \Bu \cdot \nabla ) \partial_{x_3} \Bu \cdot \Bu \varphi_R - \int_\Omega (\partial_{x_3} \Bu \cdot \nabla \varphi_R) ( \Bu \cdot \partial_{x_3} \Bu).
\ee
Note that it holds that
\be \label{6-7}
\ba
\left| \int_\Omega (\partial_{x_3} \Bu \cdot \nabla) \partial_{x_3} \Bu \cdot \Bu \varphi_R \right|
& \leq \| \nabla \partial_{x_3} \Bu \sqrt{ \varphi_R} \|_{L^2(\Omega)} \|\partial_{x_3} \Bu \sqrt{\varphi_R } \|_{L^2(\Omega)} \|\Bu \|_{L^\infty(\Omega)} \\
& \leq \|\nabla \partial_{x_3} \Bu \sqrt{ \varphi_R} \|_{L^2(\Omega)}\frac{1}{2\pi} \|\partial_{x_3}^2 \Bu \sqrt{\varphi_R } \|_{L^2(\Omega)} \|\Bu \|_{L^\infty(\Omega)} \\
& \leq \frac{1}{2\pi} \|\Bu \|_{L^\infty(\Omega)}\|\nabla \partial_{x_3} \Bu \sqrt{ \varphi_R} \|_{L^2(\Omega)}^2,
\ea
\ee
where the Wirtinger inequality (cf. \cite[p. 185]{HLP})
\[
\|\partial_{x_3} \Bu \sqrt{\varphi_R } \|_{L^2(\Omega)} \leq \frac{1}{2\pi} \|\partial_{x_3}^2 \Bu \sqrt{\varphi_R } \|_{L^2(\Omega)}
\]
has been used. Furthermore, the straightforward computations yield
and
\be \label{6-8}
 \left| \int_\Omega (\partial_{x_3} \Bu \cdot \nabla \varphi_R) (\Bu \cdot \partial_{x_3} \Bu )\right|
\leq \| \partial_{x_3} \Bu \|_{L^2(\mathcal{O}_R) }^2 \| \Bu \|_{L^\infty(\mathcal{O}_R)} \leq  CR^{\frac12}\|\nabla \partial_{x_3} \Bu\|_{L^2(\mathcal{O}_R)}
\ee
and
\be \label{6-9}
\left| \int_\Omega \Bu \cdot \nabla \varphi_R |\partial_{x_3} \Bu|^2 \right| \leq \|\Bu\|_{L^\infty(\mathcal{O}_R)}\|\partial_{x_3} \Bu\|_{L^2(\mathcal{O_R})}^2
\leq CR^{\frac12} \|\nabla \partial_{x_3} \Bu\|_{L^2(\mathcal{O}_R)}.
\ee
Finally, one has
\be \label{6-10}
\left| \int_{\Omega} \partial_{x_3} P   \partial_{x_3} \Bu  \cdot \nabla  \varphi_R \right| \leq \|\partial_{x_3} P\|_{L^2(\mathcal{O}_R)} \|\partial_{x_3} \Bu\|_{L^2(\mathcal{O}_R)} \leq C R^{\frac12} \|\nabla \partial_{x_3} \Bu\|_{L^2(\mathcal{O}_R)}, 
\ee
where the last inequality is due to the fact that $\partial_{x_3} P$ is uniformly bounded. 

Define
\be \nonumber
Z(R) = \int_{\mathbb{R}^2} \int_0^1 |\nabla \partial_{x_3} \Bu|^2\varphi_R(\sqrt{x_1^2+x_2^2}) dx_3 dx_1dx_2.
\ee
If $\|\Bu\|_{L^\infty(\Omega)} <2\pi $, one has
\be \nonumber
Z(R) \leq C R^{\frac12} Z^{\prime}(R)^{\frac12}.
\ee
The same argument as that for the proof of Theorem 1.3 proves that $\nabla \partial_{x_3} \Bu \equiv 0$, and thus
$\partial_{x_3} \Bu$ is identically equal to a constant vector, which is indeed zero because $\int_0^1 \partial_{x_3} \Bu  dx_3 =0$. 
Hence, $\Bu$ must be independent of $x_3$. Then it is straightforward to see that the pressure $P$ is
also independent of $x_3$, and thus the two-dimensional vector field $(u_1, u_2)$ is a solution to the
stationary Navier-Stokes equations in $R^2$. According to \cite[Theorem 5.1]{KNSS}, the bounded
two-dimensional solution $(u_1, u_2)$  must be a constant vector. Then by elliptic theory, one
can also show that $u_3$, which is bounded, is also a constant.

 The proof of Theorem \ref{corollary2} is completed.
	\end{proof}

	{\bf Acknowledgement.}
The research of Gui is partially supported by NSF grant DMS-1901914. The research of Wang is partially supported by NSFC grants 12171349 and 11671289. The research of  Xie is partially supported by  NSFC grants 11971307 and 11631008, and  Natural Science Foundation of Shanghai 21ZR1433300.
	
\end{document}